\documentclass[reqno]{amsart}
\usepackage{amssymb}
\newtheorem{theorem}{Theorem}[section]
\newtheorem{proposition}[theorem]{Proposition}
\newtheorem{lemma}[theorem]{Lemma}

\newtheorem{cor}[theorem]{Corollary}
\theoremstyle{definition}
\newtheorem{definition}[theorem]{Definition}
\newtheorem{example}[theorem]{Example}

\newtheorem{remark}[theorem]{Remark}

\allowdisplaybreaks
\numberwithin{equation}{section}
\begin{document}

\title[Generalized twisted Edwards curves and hypergeometric functions]
{Generalized twisted Edwards curves over finite fields and hypergeometric functions}

\author[Rupam Barman]{Rupam Barman}
\address{Department of Mathematics, Indian Institute of Technology Guwahati, North Guwahati, Assam, India, PIN- 781039
\newline 
ORCID ID: 0000-0002-4480-1788}
\email{rupam@iitg.ac.in}

\author[Sipra Maity]{Sipra Maity}
\address{Department of Mathematics, Indian Institute of Technology Guwahati, North Guwahati, Assam, India, PIN- 781039
\newline 
ORCID ID: 0009-0004-1457-394X}
%\curraddr{}
\email{s.maity@iitg.ac.in}

\author[Sulakashna]{Sulakashna}
\address{Department of Mathematics, Indian Institute of Technology Guwahati, North Guwahati, Assam, India, PIN- 781039
\newline 
ORCID ID: 0009-0008-0441-4792}
%\curraddr{}
\email{sulakash@iitg.ac.in}

\thanks{}

%    author two information

%    \subjclass is required.
\date{7th December 2024, version-1}
%\thanks{Acknowledement: We thank the referee for his/her valuable comments.}
%    Abstract is required.
\begin{abstract} 
Let $\mathbb{F}_q$ be a finite field with $q$ elements. For $a,b,c,d,e,f \in \mathbb{F}_q^{\times}$, denote by $C_{a,b,c,d,e,f}$ the family of algebraic curves over $\mathbb{F}_q$ given by the affine equation
\begin{align*}
C_{a,b,c,d,e,f}:ay^2+bx^2+cxy=d+ex^2y^2+fx^3y.
\end{align*}
The family of generalized twisted Edwards curves is a subfamily of $C_{a,b,c,d,e,f}$. Let $\#C_{a,b,c,d,e,f}(\mathbb{F}_q)$
denote the number of points on $C_{a,b,c,d,e,f}$ over $\mathbb{F}_q$.
In this article, we find certain expressions for $\#C_{a,b,c,d,e,f}(\mathbb{F}_q)$ when $af=ce$. If $c^2-4ab\neq 0$, we express $\#C_{a,b,c,d,e,f}(\mathbb{F}_q)$ in terms of a $p$-adic hypergeometric function $\mathbb{G}(x)$ whose values are explicitly known for all $x\in \mathbb{F}_q$. Next, if $c^2-4ab=0$, we express $\#C_{a,b,c,d,e,f}(\mathbb{F}_q)$ in terms of another $p$-adic hypergeometric function and then relate it to the traces of Frobenius endomorphisms of a family of elliptic curves. Furthermore, using the known values of the hypergeometric functions, we deduce some nice formulas for $\#C_{a,b,c,d,e,f}(\mathbb{F}_q)$.
\end{abstract}
\maketitle
\section{Introduction and statement of results}
Let $K$ be a field with $\text{char}(K)\neq 2$ and let $\alpha,\beta \in K $ and denote $\Delta = \alpha^2 + 4\beta$. Let $k$ and $\ell$ be elements of $K$ such that $\ell$ and $\Delta/\ell$ are non-squares in $K$. Then the generalized twisted Edwards curves with coefficients $\alpha, \beta, k$, and $\ell$ is the affine curve
	\begin{align*}
		E_G:y^2-\beta x^2+\alpha xy=k^2\{{1+\ell(y^2+\alpha x y)x^2}\}.
	\end{align*}
For more details about the generalized twisted Edwards curves, see \cite{AC}. Edwards model of elliptic curves were proposed by Edwards in \cite{H.M.E}. Since then Edwards model of eliptic curves have been used in many cryptographic applications. The group law in Edwards model is simpler to state than on other models of elliptic curves. 
%Also, any elliptic curve defined over an algebraically closed field can be expressed in the form %$x^2+y^2=a^2(1+x^2y^2)$. 
To express more elliptic curves over finite fields with the addition law being easily formulated, Bernstein et al. \cite{DB} introduced twisted Edwards curves for the first time.
The families of classical Edwards curves and twisted Edwards curves are given by the following affine equations
\begin{align*}
	&E_{a,d}:x^2+y^2=a^2(1+dx^2y^2),\\
	&E^{\prime}_{a,d}:ax^2+y^2= 1+dx^2y^2,
\end{align*}
respectively, and both these families are particular cases of the generalized twisted Edwards curves $E_G$. 
\par Let $p$ be an odd prime, and let $\mathbb{F}_q$ be the finite field containing $q$ elements, where $q=p^r,r\geq1$. The first author with Kalita \cite{Gk1} found a relation between the number of $\mathbb{F}_q$-points on twisted Edwards curves and hypergeometric functions over finite fields, and proved that the number of points on twisted Edwards curves is related to the number of points on a family of elliptic curves over $\mathbb{F}_q$. In \cite{MS}, Sadek and El-Sissi expressed the number of $\mathbb{F}_p$-points on classical Edwards curves $E_{a,1}$ and twisted Edwards curves $E^{\prime}_{a,d}$ in terms of hypergeometric functions over $\mathbb{F}_{p}$ following an approach different from \cite{Gk1}. In a recent article \cite{JPR}, Juyal et al. extended the result known for $E_{a,1}$ to a generic case of classical Edwards curves $E_{a,d}$. Also, they introduced the notion of trace of Edwards curves and related it to the trace of Frobenius endomorphism of a family of elliptic curves. 
\par
In this article, we consider a more general family of affine algebraic curves. For $a,b,c,d,e,f \in \mathbb{F}_q^{\times}$, denote by $C_{a,b,c,d,e,f}$ the family of algebraic curves over $\mathbb{F}_q$ given by the affine equation
\begin{align}\label{curve}
	C_{a,b,c,d,e,f}:ay^2+bx^2+cxy=d+ex^2y^2+fx^3y.
\end{align}
The family of generalized twisted Edwards curves $E_G$ is a subfamily of $C_{a,b,c,d,e,f}$. Very recently, Azharuddin and Kalita \cite{AS} found a relation between the finite field Appell series and the number of points on the affine curve defined in \eqref{curve} when $f=0$. In this article, we find certain expressions for the number of $\mathbb{F}_q$-points on $C_{a,b,c,d,e,f}$ in terms of $p$-adic hypergeometric functions. 
\par To state our main results, we now introduce $p$-adic hypergeometric functions. Let $\mathbb{Z}_p$ and $\mathbb{Q}_p$ denote the ring of $p$-adic integers and the field of $p$-adic numbers, respectively. Let $\Gamma_p(\cdot)$ denote the Morita's $p$-adic gamma function, and let $\omega$ denote the 
Teichm\"{u}ller character of $\mathbb{F}_q$. We denote by $\overline{\omega}$ the character inverse of $\omega$. 
For $x \in \mathbb{Q}$, we let $\lfloor x\rfloor$ denote the greatest integer less than or equal to $x$ and $\langle x\rangle$ denote the fractional part of $x$, i.e., $x-\lfloor x\rfloor$, satisfying $0\leq\langle x\rangle<1$. McCarthy's $p$-adic hypergeometric function $_{n}G_{n}[\cdots]_q$ is defined as follows.
\begin{definition}(\emph{\cite[Definition 5.1]{mccarthy2}}). \label{defin1}
	Let $p$ be an odd prime and $q=p^r$, $r\geq 1$. Let $t \in \mathbb{F}_q$.
	For positive integer $n$ and $1\leq k\leq n$, let $a_k$, $b_k$ $\in \mathbb{Q}\cap \mathbb{Z}_p$.
	Then the function $_{n}G_{n}[\cdots]_q$ is defined by
	\begin{align}
		&_nG_n\left[\begin{array}{cccc}
			a_1, & a_2, & \ldots, & a_n \\
			b_1, & b_2, & \ldots, & b_n
		\end{array}|t
		\right]_q\notag\\
		&\hspace{1cm}:=\frac{-1}{q-1}\sum_{a=0}^{q-2}(-1)^{an}~~\overline{\omega}^a(t)
		\prod\limits_{k=1}^n\prod\limits_{i=0}^{r-1}(-p)^{-\lfloor \langle a_kp^i \rangle-\frac{ap^i}{q-1} \rfloor -\lfloor\langle -b_kp^i \rangle +\frac{ap^i}{q-1}\rfloor}\notag\\
		&\hspace{2cm} \times \frac{\Gamma_p(\langle (a_k-\frac{a}{q-1})p^i\rangle)}{\Gamma_p(\langle a_kp^i \rangle)}
		\frac{\Gamma_p(\langle (-b_k+\frac{a}{q-1})p^i \rangle)}{\Gamma_p(\langle -b_kp^i \rangle)}.\notag
	\end{align}
\end{definition}
For $x\in \mathbb{F}_q$, we denote by $\mathbb{G}(x)$ the following $p$-adic hypergeometric function
\begin{align*}
	\mathbb{G}(x):={_2}G_2\left[\begin{array}{cc}
		\frac{1}{4}, & \frac{3}{4} \vspace*{0.05cm}\\
		0, & \frac{1}{2}
	\end{array}|x
	\right]_q.
\end{align*}
In general, finding special values of $p$-adic hypergeometric functions is a difficult problem. However, the values of $\mathbb{G}(x)$ are known for all $x$ due to Saikia \cite{NS}. 
\begin{theorem}\emph{(\cite[Theorem 1.2]{NS})}.\label{sv}
	Let $p\geq3$ be a prime and $q=p^r$, $r\geq1$. Let  $t\in\mathbb{F}_q^\times$. Then 
	\begin{enumerate}
		\item
		\begin{align*}
			\mathbb{G}(1)=
			\left\{\begin{array}{ll}
				\hspace*{0.27cm}1, & \hbox{if $q\equiv\pm1\pmod8$;} \\
				-1, & \hbox{if $q\equiv\pm3\pmod8$.}
			\end{array}
			\right.
		\end{align*}	
		\item Let $t\neq1$ and $\frac{t-1}{t}$ be a square in $\mathbb{F}_q^\times$ such that $\frac{t-1}{t}=a^2$ for some $a\in\mathbb{F}_q^\times$. Then we have
		\begin{align*}
			\mathbb{G}(t)=\varphi(2)(\varphi(1+a)+\varphi(1-a)).
		\end{align*}
		\item If $\frac{t-1}{t}$ is not a square in $\mathbb{F}_q$, then
		\begin{align*}
			\mathbb{G}(t)=0.
		\end{align*}
	\end{enumerate}
\end{theorem}
\begin{remark}
	We note that Theorem \ref{sv} is proved by Saikia for $\mathbb{F}_p$, where $p$ is an odd prime. However, she remarked that the theorem is also true for $\mathbb{F}_q$, and we have also verified the same.
\end{remark}
\par Throughout the paper, $p$ is an odd prime, and $\mathbb{F}_q$ denotes the finite field containing $q$ elements, where $q=p^r,r\geq1$. Also, let $\widehat{\mathbb{F}_q^{\times}}$ be the group of all the multiplicative characters on $\mathbb{F}_q^{\times}$. We extend the domain of each $\chi\in \widehat{\mathbb{F}_q^{\times}}$ to $\mathbb{F}_q$ by setting $\chi(0):=0$
including the trivial character $\varepsilon$. Let $\varphi$ denote the quadratic character on $\mathbb{F}_q$. Let $\delta$ denote the function on $\widehat{\mathbb{F}_q^\times}$ defined by
\begin{align}\label{delta-1}
	\delta(A):=\left\{
	\begin{array}{ll}
		1, & \hbox{if $A=\varepsilon$;} \\
		0, & \hbox{otherwise.}
	\end{array}
	\right.
\end{align}
We also denote by $\delta$ the function defined on $\mathbb{F}_q$ by
\begin{align*}
	\delta(x):=\left\{
	\begin{array}{ll}
		1, & \hbox{if $x=0$;} \\
		0, & \hbox{otherwise.}
	\end{array}
	\right.
\end{align*}
Our first main result gives an expression for the number of $\mathbb{F}_q$-points on $C_{a,b,c,d,e,f}$ in terms of the $p$-adic hypergeometric functions $\mathbb{G}(x)$. 
\begin{theorem}\label{MT-1}
	Let $p$ be an odd prime. Let $\chi_4$ be a multiplicative character of order $4$ on $\mathbb{F}_q$. Let $a,b,c,d,e,f\in\mathbb{F}_q^\times$ be such that $af=ce$ and $c^2-4ab\neq0$. Let $\#C_{a,b,c,d,e,f}(\mathbb{F}_q)$ denote the number of $\mathbb{F}_q$-points on the affine algebraic curve $C_{a,b,c,d,e,f}$. Then we have 
	\begin{align*}
		\#C_{a,b,c,d,e,f}(\mathbb{F}_q)=q-2+\varphi(bd)+\varphi(ad)-\varphi(ae)+\varphi(abde(c^2-4ab))+I_2^\prime\\ +\varphi(c^2-4ab)\sum_{y\in\mathbb{F}_q}\varphi(y(y-1)(1+u_y))+q\delta\left(1-\frac{ab}{de}\right)(1+\varphi(ae))+X,
	\end{align*}
	where
	\begin{align*}
		&u_y=\frac{16de(1-y)(c^2-4aby)}{(c^2-4ab)^2},\\
		&X=\varphi(ad)\sum_{y\in\mathbb{F}_q\backslash\{1,\frac{c^2}{4ab}\} }\varphi(y)\mathbb{G}\left(-\frac{1}{u_y}\right),\\
		&I_2^\prime=\left\{
		\begin{array}{ll}
			\varphi(ae(c^2-4ab))\left(\chi_4\left(\frac{de}{ab}\right)+\overline{\chi_4}\left(\frac{de}{ab}\right)\right), & \hbox{if $q\equiv1\pmod4$;} \\
			0, & \hbox{if $q\equiv3\pmod4$.}
		\end{array}
		\right.
	\end{align*}
\end{theorem}
Since the values of $\mathbb{G}(x)$ are known, by putting some extra conditions on the coefficients we obtain nice formulas for $\#C_{a,b,c,d,e,f}(\mathbb{F}_q)$. For example, we have the following corollary  from Theorem \ref{MT-1}.
\begin{cor}\label{special-case}
	Let $p$ be an odd prime. Let $a,b,c,d,e,f\in\mathbb{F}_q^\times$ be such that $af=ce$, $c^2=-4ab$, and $ab=de$. Then we have
	\begin{align*}
	\#C_{a,b,c,d,e,f}(\mathbb{F}_q)=\left\{
	\begin{array}{ll}
	2q-2+(q-1) \varphi(ad), & \hbox{if $q\equiv1\pmod4$;}\\
	2q-2-(q-1)\varphi(ad), & \hbox{if $q\equiv3\pmod4$.}
	\end{array}
	\right.
	\end{align*}
\end{cor}
\begin{example}
	Here we take $a=d=1$, $b=e=-1$, $c=2$ and $f=-2$. Then, from Corollary \ref{special-case}, we have
	\begin{align*}
	&\#\{(x, y)\in \mathbb{F}_q^2: y^2-x^2+2xy=1-x^2y^2-2x^3y\}\\
	&=\left\{
	\begin{array}{ll}
	3(q-1), & \hbox{if $q\equiv1\pmod4$;}\\
	q-1, & \hbox{if $q\equiv3\pmod4$.}
	\end{array}
	\right.
	\end{align*}
	If we take $a=e=1$, $b=d=-1$, and $c=f=2$, then from Corollary \ref{special-case}, we have
	\begin{align*}
	&\#\{(x, y)\in \mathbb{F}_q^2: y^2-x^2+2xy=-1+x^2y^2+2x^3y\}\\
	&=\left\{
	\begin{array}{ll}
	q-1, & \hbox{if $q\equiv1\pmod4$;}\\
	3(q-1), & \hbox{if $q\equiv3\pmod4$.}
	\end{array}
	\right.
	\end{align*}
\end{example}
The family of generalized twisted Edwards curves is a subfamily of $C_{a,b,c,d,e,f}$, and hence we obtain the following corollary which gives a relation between the number of $\mathbb{F}_{q}$-points on the generalized twisted Edwards curves and the special values of the $p$-adic hypergeometric functions $\mathbb{G}(x)$.
\begin{cor}\label{cor-1}
	Let $p$ be an odd prime. Let $\chi_4$ be a multiplicative character of order $4$ on $\mathbb{F}_q$. Let $\alpha,\beta,k,\ell\in\mathbb{F}_q^\times$ be such that $\Delta=\alpha^2+4\beta\neq0$, and $\ell$ and $\Delta/\ell$ are non-squares. If $\#E_G(\mathbb{F}_q)$ denotes the number of affine points on the generalized twisted Edwards curves $$E_G: y^2-\beta x^2+\alpha xy=k^2\{{1+\ell(y^2+\alpha x y)x^2}\}$$ over $\mathbb{F}_q$, then 
	\begin{align*}
		\#E_G(\mathbb{F}_q)=\sum_{y\in\mathbb{F}_q\backslash\{1,\frac{-\alpha^2}{4\beta}\}}\varphi(y)\mathbb{G}\left(-\frac{1}{v_y}
		\right)+\sum_{y\in\mathbb{F}_q}\varphi(y(y-1)(1+v_y))+J_2,
	\end{align*}
	where
	\begin{align*}
		&v_y=\frac{16k^4\ell(1-y)(\alpha^2+4\beta y)}{(\alpha^2+4\beta)^2},\\
		&J_2=\left\{
		\begin{array}{ll}
			q-\chi_4\left(\frac{-\ell}{\beta}\right)-\overline{\chi_4}\left(\frac{-\ell}{\beta}\right), & \hbox{if $q\equiv1\pmod4$;} \\
			q, & \hbox{if $q\equiv3\pmod4$.}
		\end{array}
		\right.
	\end{align*}
\end{cor}
In Theorem \ref{MT-1}, we considered the algebraic curve $C_{a,b,c,d,e,f}$ with the condition that $c^2-4ab\neq0$. In the next theorem, we find an expression for $\#C_{a,b,c,d,e,f}(\mathbb{F}_q)$ in terms of the special values of a $p$-adic hypergeometric function under the condition that $c^2-4ab=0$. Interestingly, the $p$-adic hypergeometric function appearing in the expression for $\#C_{a,b,c,d,e,f}(\mathbb{F}_q)$ is different from $\mathbb{G}(x)$. More precisely, we have the following theorem.
\begin{theorem}\label{MT-6}
		Let $p$ be an odd prime. Let $a,b,c,d,e,f\in\mathbb{F}_q^\times$ be such that $af=ce$ and $c^2-4ab=0$. Let $\#C_{a,b,c,d,e,f}(\mathbb{F}_q)$ denote the number of $\mathbb{F}_q$-points on the affine algebraic curve $C_{a,b,c,d,e,f}$. Then we have
	\begin{align*}
		\#C_{a,b,c,d,e,f}(\mathbb{F}_q)&=q-1-q\varphi(ad)\cdot {_2}G_2\left[\begin{array}{cc}
			\frac{1}{2}, & \frac{1}{2}\vspace*{0.05cm}\\
			\frac{1}{4}, & \frac{3}{4}
		\end{array}|\frac{4de}{c^2}
		\right]_q-q\cdot {_2}G_2\left[\begin{array}{cc}
			\frac{1}{2}, & \frac{1}{2} \vspace*{0.05cm}\\
			\frac{1}{4}, & \frac{3}{4}
		\end{array}|\frac{c^2}{4de}
		\right]_q\\
		&\hspace{1cm} + \left(q\delta\left(1-\frac{ab}{de}\right)-1\right)(1+\varphi(ae)).
	\end{align*}
\end{theorem}
Unlike to the hypergeometric function $\mathbb{G}(x)$, the values of the hypergeometric function appearing in Theorem \ref{MT-6} are not known for all $x\in \mathbb{F}_q$. However, certain specific values are known, and by putting some extra conditions on the coefficients we obtain nice formulas for $\#C_{a,b,c,d,e,f}(\mathbb{F}_q)$. For example, we have the following corollary  from Theorem \ref{MT-6}.
\begin{cor}\label{new_cor-1}
	Let $p$ be an odd prime. Let $q=p^r$ be such that $q\equiv1\pmod4$. Suppose that $a,b,c,d,e,f\in\mathbb{F}_q^\times$ satisfy $af=ce$, $c^2=4ab$ and $ab=de$. Then we have
	\begin{align*}
		\#C_{a,b,c,d,e,f}(\mathbb{F}_q)&=\left\{
		\begin{array}{ll}
			2q-3+(q-2)\varphi(ad), & \hbox{if $q\equiv1\pmod8$;} \\
			2q-1+q\varphi(ad), & \hbox{if $q\equiv5\pmod8$.}
		\end{array}
		\right.
	\end{align*}
\end{cor}
\begin{example}
	If we take $a=b=d=e=1$ and $c=f=2$, then from Corollary \ref{new_cor-1} we have
	\begin{align*}
		\#\{(x,y): x^2+y^2-2xy=1+x^2y^2-2x^3y\}=\left\{\begin{array}{ll}
			3q-5, & \hbox{if $q\equiv1\pmod8$;}\\
			3q-1, & \hbox{if $q\equiv5\pmod8$.}
		\end{array}\right.
	\end{align*}
\end{example}
Let $E$ be an elliptic curve given in the Weierstrass form over the finite field $\mathbb{F}_q$. Then the trace of Frobenius $a_q(E)$ of $E$ is given by 
$$a_q(E):=q+1-\#E(\mathbb{F}_q),$$
 where $\#E(\mathbb{F}_q)$ denotes the number of $\mathbb{F}_q$-points on $E$ including the point at infinity. As a corollary of Theorem \ref{MT-6}, we express $\#C_{a,b,c,d,e,f}(\mathbb{F}_q)$ in terms of traces of Frobenius of a family of elliptic curves.
\begin{cor}\label{cor-2}
Let $p$ be an odd prime. Let $a,b,c,d,e,f\in\mathbb{F}_q^\times$ be such that $af=ce$ and $c^2-4ab=0$. Let $\#C_{a,b,c,d,e,f}(\mathbb{F}_q)$ denote the number of $\mathbb{F}_q$-points on the affine algebraic curve $C_{a,b,c,d,e,f}$. For $h,g\in\mathbb{F}_q^\times$, let $E_{h,g}$ denote the elliptic curve given by $y^2=x^3+hx^2+gx$. If $c^2g=h^2de$, then we have
\begin{align*}
		\#C_{a,b,c,d,e,f}(\mathbb{F}_q) & =q-1-\varphi(-aeh)a_q(E_{h,g})-\varphi(-hg)a_q(E_{\frac{4}{h},\frac{1}{g}})\\
		& \hspace{1cm} + \left(q\delta\left(1-\frac{ab}{de}\right)-1\right)(1+\varphi(ae)).
\end{align*}
\end{cor}
The rest of this paper is organized as follows. In Section 2, we introduce some definitions and recall some
properties of multiplicative characters, Gauss sums, and the $p$-adic gamma function. We also
recall hypergeometric functions over finite fields as introduced by Greene. In Section 3, we present the proof of Theorem \ref{MT-1} and in Section 4, we present the proof of Theorem \ref{MT-6}.
\section{Preliminaries} 
For multiplicative characters $A$ and $B$ on $\mathbb{F}_q$,
the binomial coefficient ${A \choose B}$ is defined by
\begin{align}\label{eq-0}
	{A \choose B}:=\frac{B(-1)}{q}J(A,\overline{B})=\frac{B(-1)}{q}\sum_{x \in \mathbb{F}_q}A(x)\overline{B}(1-x),
\end{align}
where $J(A, B)$ denotes the Jacobi sum and $\overline{B}$ is the character inverse of $B$. We recall the following properties of the binomial coefficients from \cite{greene}:
\begin{align}\label{eq-4}
	{A\choose \varepsilon}={A\choose A}=\frac{-1}{q}+\frac{q-1}{q}\delta(A).
\end{align}
Using \eqref{eq-0} and \eqref{eq-4}, we obtain certain values of Jacobi sums as follows
$$J(\varepsilon,\varepsilon)=q-2, \ \ J(\varphi,\varepsilon)=-1, \ \ J(\varphi,\varphi)=-\varphi(-1). $$
The orthogonality relation satisfied by multiplicative characters is given as follows:
\begin{align}\label{eq-1}
	\sum\limits_{x\in {\mathbb{F}_q }} \chi(x)= \left\{
	\begin{array}{ll}
		q-1, & \hbox{if $\chi=\varepsilon$;} \\
		0, & \hbox{otherwise.}
	\end{array}
	\right.
\end{align}
We shall also use the following orthogonality relation for characters:
\begin{align}\label{eq-2}
	\sum_{\chi\in \widehat{\mathbb{F}_q^{\times}}} \chi(x)= \left\{
	\begin{array}{ll}
		q-1, & \hbox{if $x=1$;} \\
		0, & \hbox{otherwise.}
	\end{array}
	\right.
\end{align}
\par
Let $\overline{\mathbb{Q}_p}$ be the algebraic closure of $\mathbb{Q}_p$ and $\mathbb{C}_p$ be the completion of $\overline{\mathbb{Q}_p}$.
Let $\mathbb{Z}_q$ be the ring of integers in the unique unramified extension of $\mathbb{Q}_p$ with residue field $\mathbb{F}_q$.
We know that $\chi\in \widehat{\mathbb{F}_q^{\times}}$ takes values in $\mu_{q-1}$, where $\mu_{q-1}$ is the group of all the $(q-1)$-th roots of unity in $\mathbb{C}^{\times}$. Since $\mathbb{Z}_q^{\times}$ contains all the $(q-1)$-th roots of unity,
we can consider multiplicative characters on $\mathbb{F}_q^\times$
to be maps $\chi: \mathbb{F}_q^{\times} \rightarrow \mathbb{Z}_q^{\times}$.
Let $\omega: \mathbb{F}_q^\times \rightarrow \mathbb{Z}_q^{\times}$ be the Teichm\"{u}ller character.
For $a\in\mathbb{F}_q^\times$, the value $\omega(a)$ is just the $(q-1)$-th root of unity in $\mathbb{Z}_q$ such that $\omega(a)\equiv a \pmod{p}$.
\par Next, we introduce the Gauss sum and recall some of its elementary properties. For further details, see \cite{evans}. Let $\zeta_p$ be a fixed primitive $p$-th root of unity
in $\overline{\mathbb{Q}_p}$. The trace map $\text{tr}: \mathbb{F}_q \rightarrow \mathbb{F}_p$ is given by
\begin{align}
	\text{tr}(\alpha)=\alpha + \alpha^p + \alpha^{p^2}+ \cdots + \alpha^{p^{r-1}}.\notag
\end{align}
Then the additive character
$\theta: \mathbb{F}_q \rightarrow \mathbb{Q}_p(\zeta_p)$ is defined by
\begin{align}
	\theta(\alpha)=\zeta_p^{\text{tr}(\alpha)}.\notag
\end{align}
For $\chi \in \widehat{\mathbb{F}_q^\times}$, the \emph{Gauss sum} is defined by
\begin{align}
	g(\chi):=\sum\limits_{x\in \mathbb{F}_q}\chi(x)\theta(x) .\notag
\end{align}
\begin{lemma}\emph{(\cite[Lemma 2.2]{fuselier}).}\label{lemma2_02} For $\alpha\in \mathbb{F}^{\times}_q$, we have
	\begin{align}
		\theta(\alpha)=\frac{1}{q-1}\sum_{\chi\in\widehat{\mathbb{F}_q^\times}} g(\overline{\chi})\chi(\alpha).\notag
	\end{align}
\end{lemma}
\begin{lemma}\emph{(\cite[Eq. 1.12]{greene}).}\label{lemma2_1}
	For $\chi \in \widehat{\mathbb{F}_q^\times}$, we have
	$$g(\chi)g(\overline{\chi})=q\cdot \chi(-1)-(q-1)\delta(\chi).$$
\end{lemma}
The following lemma gives a relation between Jacobi and Gauss sums.
\begin{lemma}\emph{(\cite[Eq. 1.14]{greene}).}\label{lemma2_2} For $A,B\in\widehat{\mathbb{F}_q^{\times}}$, we have
	\begin{align}
		J(A,B)=\frac{g(A)g(B)}{g(AB)}+(q-1)B(-1)\delta(AB).\notag
	\end{align}
\end{lemma}
\begin{theorem}\emph{(\cite[Davenport-Hasse Relation]{evans}).}\label{thm2_2}
	Let $m$ be a positive integer and let $q=p^r$ be a prime power such that $q\equiv 1 \pmod{m}$. For multiplicative characters
	$\chi, A\in \widehat{\mathbb{F}_q^\times}$, we have
	\begin{align}
		\prod\limits_{\chi^m=\varepsilon}g(A\chi )=-g(A^m)A(m^{-m})\prod\limits_{\chi^m=\varepsilon}g(\chi).\notag
	\end{align}
\end{theorem}
In \cite{greene2, greene}, Greene introduced the notion of hypergeometric series over finite fields famously known as \emph{Gaussian hypergeometric series}. He defined the $_2F_1(\cdots)$ hypergeometric functions over finite fields as follows.
\begin{definition}(\emph{\cite[Definition 3.5]{greene}}).\label{char-def}
	For characters $A, B$, and $C$ on $\mathbb{F}_q$ and $x\in\mathbb{F}_q$,
	\begin{align*}
		&_{2}F_1\left(\begin{array}{cc}
			A & B \\
			& C
		\end{array}|x
		\right):= \varepsilon(x)\frac{BC(-1)}{q}\sum_{y\in\mathbb{F}_q}B(y)\overline{B}C(1-y)\overline{A}(1-xy).
	\end{align*}
\end{definition}
Using the relation between Jacobi sum and binomial coefficients, Greene found an expression for $_2F_1(\cdots)$ in terms of binomial coefficients.
\begin{theorem}\emph{(\cite[Theorem 3.6]{greene})}.\label{greene-def}
	For characters $A, B$, and $C$ on $\mathbb{F}_q$ and $x\in\mathbb{F}_q$,
	\begin{align*}
	{_{2}}F_1\left(\begin{array}{cc}
	A & B\\
	& C
	\end{array}\mid x \right)=\frac{q}{q-1}\sum_{\chi\in\widehat{\mathbb{F}_{q}^{\times}}}{A\chi \choose \chi}{B\chi \choose C\chi}\chi(x).
	\end{align*}
\end{theorem}
We now recall a theorem that helps us to write a character sum involving Gauss sums in terms of a quotient of Gauss sums.
\begin{theorem}\emph{(\cite[Theorem 2]{HP})}.\label{thrm--3}
	For characters  $A, B, C$, and $D$ on $\mathbb{F}_q$, we have
	\begin{align*}
		&\frac{1}{q-1}\sum_{\chi\in\widehat{\mathbb{F}_{q}^{\times}}}g(A\chi)g(B\overline{\chi})g(C\chi)g(D\overline{\chi})\\
		&\hspace*{2.5cm}=\frac{g(AB)g(AD)g(BC)g(CD)}{g(ABCD)}+q(q-1)AC(-1)\delta(ABCD).
	\end{align*}
\end{theorem}
We also recall the following transformation identity for Gaussian hypergeometric series.
\begin{theorem}\emph{(\cite[Theorem 4.4 (iv)]{greene})}.\label{thrm--4}
		For characters $A, B$, and $C$ on $\mathbb{F}_q$ and $x\in\mathbb{F}_q$,
		\begin{align*}
			_{2}F_1\left(\begin{array}{cc}
				A & B \\
				& C
			\end{array}|x
			\right)=C(-1)\overline{AB}C(1-x){_2}F_1\left(\begin{array}{cc}
					\overline{A}C & \overline{B}C \\
					& C
				\end{array}|x
				\right)\\
				+ A(-1){B\choose \overline{A}C}\delta(1-x).		
		\end{align*}
\end{theorem}
Next, we introduce the $p$-adic gamma function. For further details, see \cite{kob}.
For a positive integer $n$,
the $p$-adic gamma function $\Gamma_p(n)$ is defined as
\begin{align}
	\Gamma_p(n):=(-1)^n\prod\limits_{0<j<n,p\nmid j}j\notag
\end{align}
and one extends it to all $x\in\mathbb{Z}_p$ by setting $\Gamma_p(0):=1$ and
\begin{align}
	\Gamma_p(x):=\lim_{x_n\rightarrow x}\Gamma_p(x_n)\notag
\end{align}
for $x\neq0$, where $x_n$ runs through any sequence of positive integers $p$-adically approaching $x$.
This limit exists, is independent of how $x_n$ approaches $x$,
and determines a continuous function on $\mathbb{Z}_p$ with values in $\mathbb{Z}_p^{\times}$.
Let $\pi \in \mathbb{C}_p$ be the fixed root of $x^{p-1} + p=0$ which satisfies
$\pi \equiv \zeta_p-1 \pmod{(\zeta_p-1)^2}$. Then the Gross-Koblitz formula relates Gauss sums and the $p$-adic gamma function as follows.
\begin{theorem}\emph{(\cite[Gross-Koblitz]{gross}).}\label{thm2_3} 
	For $a\in \mathbb{Z}$ and $q=p^r, r\geq 1$, we have
	\begin{align}
		g(\overline{\omega}^a)=-\pi^{(p-1)\sum\limits_{i=0}^{r-1}\langle\frac{ap^i}{q-1} \rangle}\prod\limits_{i=0}^{r-1}\Gamma_p\left(\left\langle \frac{ap^i}{q-1} \right\rangle\right).\notag
	\end{align}
\end{theorem}
The following lemmas relate certain products of values of the $p$-adic gamma function.
\begin{lemma}\emph{(\cite[Lemma 3.1]{BS1}).}\label{lemma-3_1}
	Let $p$ be a prime and $q=p^r, r\geq 1$. For $0\leq j\leq q-2$ and $t\geq 1$ with $p\nmid t$, we have
	\begin{align}
		\omega(t^{-tj})\prod\limits_{i=0}^{r-1}\Gamma_p\left(\left\langle\frac{-tp^ij}{q-1}\right\rangle\right)
		\prod\limits_{h=1}^{t-1}\Gamma_p\left(\left\langle \frac{hp^i}{t}\right\rangle\right)
		=\prod\limits_{i=0}^{r-1}\prod\limits_{h=0}^{t-1}\Gamma_p\left(\left\langle\frac{p^i(1+h)}{t}-\frac{p^ij}{q-1}\right\rangle \right).\notag
	\end{align}
\end{lemma}
\begin{lemma}\emph{(\cite[Lemma 3.1]{BS1})}\label{lemma3_2}
	Let $p$ be a prime and $q=p^r, r\geq 1$. For $0\leq j\leq q-2$ and $t\geq 1$ with $p\nmid t$, we have
	\begin{align*}
		\omega(t^{tj})\prod\limits_{i=0}^{r-1}\Gamma_p\left(\hspace{-.1cm} \left\langle\frac{tp^ij}{q-1}\right\rangle\hspace{-.1cm}\right)
		\prod\limits_{h=1}^{t-1}\Gamma_p\left(\hspace{-.1cm}\left\langle \frac{hp^i}{t}\right\rangle\hspace{-.1cm}\right)
		=\prod\limits_{i=0}^{r-1}\prod\limits_{h=0}^{t-1}\Gamma_p\left(\hspace{-.1cm} \left\langle\frac{p^i h}{t}+\frac{p^ij}{q-1}\right\rangle \hspace{-.1cm}\right).\notag
	\end{align*}
\end{lemma}
Next, we recall certain lemmas relating fractional and integral parts of certain rational numbers.
\begin{lemma}\emph{(\cite[Lemma 2.6]{SB}).}\label{lemma-3_3}
	Let $p$ be an odd prime and $q=p^r, r\geq 1$. Let $d\geq2$ be an integer such that $p\nmid d$. Then, for $1\leq j\leq q-2$ and $0\leq i\leq r-1$, we have
	\begin{align*}
	\left\lfloor\frac{jp^i}{q-1}\right\rfloor +\left\lfloor\frac{-djp^i}{q-1}\right\rfloor = \sum_{h=1}^{d-1} \left\lfloor\left\langle\frac{hp^i}{d}\right\rangle-\frac{jp^i}{q-1}\right\rfloor -1.
	\end{align*}
\end{lemma}
\begin{lemma}\emph{(\cite[Lemma 2.7]{SB}).}\label{lemma-3_4}
	Let $p$ be an odd prime and $q=p^r, r\geq 1$. Let $l$ be a positive integer such that $p\nmid l$. Then, for $0\leq j\leq q-2$ and $0\leq i\leq r-1$, we have
	\begin{align*}
		\left\lfloor\frac{ljp^i}{q-1}\right\rfloor = \sum_{h=0}^{l-1} \left\lfloor\left\langle\frac{-hp^i}{l}\right\rangle+\frac{jp^i}{q-1}\right\rfloor.
	\end{align*}
\end{lemma}
Finally, we recall a result of the first author and Saikia \cite{BS1} where they found a relation between the trace of Frobenius of a family of elliptic curves and the special value of a $p$-adic hypergeometric function. Let $E_{h,g}$ be the family of elliptic curves over $\mathbb{F}_q$ given by
\begin{align*}
E_{h,g}:y^2=x^3+hx^2+gx, h\neq 0.
\end{align*}
\begin{theorem}\emph{(\cite[Theorem 3.5]{BS1}).}\label{thrm--5}
	Let $p$ be an odd prime and $q=p^r,r\geq1$. The trace of Frobenius on $E_{h,g}$ is given  by
	\begin{align*}
		a_q(E_{h,g})=q\cdot\varphi(-hg)\cdot{_2}G_{2}\left[\begin{array}{cc}
			\frac{1}{2},  & \frac{1}{2}\vspace*{0.05cm} \\
			\frac{1}{4}, & \frac{3}{4}
		\end{array}|\frac{4g}{h^2}
		\right]_q.
	\end{align*}
\end{theorem}
\section{Proof of Theorem \ref{MT-1}}
Firstly, we prove a proposition which will be used in the proof of Theorem \ref{MT-1}.
\begin{proposition}\label{prop-1}
	Let $p$ be an odd prime and $q=p^r$, $r\geq1$. Then, for $x\in\mathbb{F}_q^\times$, we have
	\begin{align*}
	\sum_{ \psi\in\widehat{\mathbb{F}_{q}^{\times}}}g(\overline{\psi})g(\overline{\psi}\varphi)g(\psi^2\varphi)\psi\left(\frac{x}{4}\right)=q(q-1)\varphi(-x)\cdot{_2}G_2\left[\begin{array}{cc}
	\frac{1}{4}, & \frac{3}{4}\vspace*{0.05cm} \\
	0, & \frac{1}{2}
	\end{array}|\frac{1}{x}
	\right]_q.
	\end{align*}
\end{proposition}
\begin{proof}
	Let 
		\begin{align*}
	E_3:= 	\sum_{ \psi\in\widehat{\mathbb{F}_{q}^{\times}}}g(\overline{\psi})g(\overline{\psi}\varphi)g(\psi^2\varphi)\psi\left(\frac{x}{4}\right).
	\end{align*}
The change of variable $\psi\rightarrow\psi\varphi$ yields
\begin{align}\label{eq--14}
	E_3= 	\varphi(x)\sum_{ \psi\in\widehat{\mathbb{F}_{q}^{\times}}}g(\overline{\psi})g(\overline{\psi}\varphi)g(\psi^2\varphi)\psi\left(\frac{x}{4}\right).
\end{align}
Using Davenport-Hasse relation with $m=2$ and $A=\psi^2$, we have
\begin{align}\label{eq--3}
g(\psi^2\varphi)=\frac{	g(\psi^4)g(\varphi)\overline{\psi}^2(4)}{g(\psi^2)}.
\end{align}
Substituting \eqref{eq--3} in \eqref{eq--14} and replacing $\psi$ with $\omega^j$ yield
	\begin{align*}
		E_3=\varphi(x) \sum_{j=0}^{q-2}g(\overline{\omega}^j)g(\overline{\omega}^{j+\frac{q-1}{2}})\frac{g(\overline{\omega}^{-4j})g(\overline{\omega}^{\frac{q-1}{2}})}{g(\overline{\omega}^{-2j})}\overline{\omega}^j\left(\frac{64}{x}\right).
	\end{align*}
Using Gross-Koblitz formula, we deduce that
\begin{align}\label{eq--2}
	E_3=-\varphi(x)\sum_{j=0}^{q-2}(-p)^{\alpha_j}M_j\overline{\omega}^j\left(\frac{64}{x}\right),
\end{align}
where
\begin{align}\label{eq--6}
	\alpha_j&:=\sum_{i=0}^{r-1}\left(\left\langle\frac{jp^i}{q-1}\right\rangle+\left\langle\frac{jp^i}{q-1}+\frac{p^i}{2}\right\rangle+\left\langle\frac{-4jp^i}{q-1}\right\rangle+\left\langle\frac{p^i}{2}\right\rangle-\left\langle\frac{-2jp^i}{q-1}\right\rangle\right),\notag\\
	M_j&:=\prod_{i=0}^{r-1}\frac{\Gamma_p\left(\left\langle\frac{jp^i}{q-1}\right\rangle\right)\Gamma_p\left(\left\langle\frac{p^i}{2}+\frac{jp^i}{q-1}\right\rangle\right)\Gamma_p\left(\left\langle\frac{-4jp^i}{q-1}\right\rangle\right)\Gamma_p\left(\left\langle\frac{p^i}{2}\right\rangle\right)}{\Gamma_p\left(\left\langle\frac{-2jp^i}{q-1}\right\rangle\right)}.
\end{align}
We have $\langle x+n\rangle=\langle x\rangle$ and $\langle x\rangle=x-\lfloor x\rfloor$ for all $x\in\mathbb{R}$ and $n\in\mathbb{Z}$. Thus, we obtain 
\begin{align}\label{eq--4}	\alpha_j&=\sum_{i=0}^{r-1}\left(-\left\lfloor\frac{jp^i}{q-1}\right\rfloor-\left\lfloor\frac{jp^i}{q-1}+\frac{1}{2}\right\rfloor-\left\lfloor\frac{-4jp^i}{q-1}\right\rfloor+\left\lfloor\frac{-2jp^i}{q-1}\right\rfloor+1\right)\nonumber\\
	&=\sum_{i=0}^{r-1}\left(-\left\lfloor\frac{jp^i}{q-1}\right\rfloor-\left\lfloor\frac{jp^i}{q-1}+\left\langle\frac{-p^i}{2}\right\rangle\right\rfloor-\left\lfloor\frac{-4jp^i}{q-1}\right\rfloor+\left\lfloor\frac{-2jp^i}{q-1}\right\rfloor+1\right).
\end{align}
For $1\leq j\leq q-2$, using Lemma \ref{lemma-3_3} with $d=4$ and $d=2$, we deduce that
\begin{align*}
	\alpha_j=r+\beta_j,
\end{align*}
where
\begin{align}\label{eq--5}
	\beta_j:=\sum_{i=0}^{r-1} \left(-\left\lfloor\frac{jp^i}{q-1}\right\rfloor-\left\lfloor\left\langle\frac{-p^i}{2}\right\rangle+\frac{jp^i}{q-1}\right\rfloor-\left\lfloor\left\langle\frac{p^i}{4}\right\rangle-\frac{jp^i}{q-1}\right\rfloor-\left\lfloor\left\langle\frac{3p^i}{4}\right\rangle-\frac{jp^i}{q-1}\right\rfloor\right).
\end{align}
Putting $j=0$ in \eqref{eq--4} and \eqref{eq--5}, we obtain $\alpha_0=r$ and $\beta_0=0$, respectively. Therefore, $\alpha_j=r+\beta_j$ for $0\leq j\leq q-2$.
We now use Lemma \ref{lemma-3_1} with $t=4$ and $t=2$ in \eqref{eq--6} and deduce that
$$M_j=\overline{\omega}^j(4^{-3})N_j\prod_{i=0}^{r-1}\Gamma_p\left(\left\langle\frac{p^i}{2}\right\rangle\right)^2,$$
where
\begin{align*}
	N_j:=\prod_{i=0}^{r-1}\frac{\Gamma_p\left(\left\langle\frac{jp^i}{q-1}\right\rangle\right)\Gamma_p\left(\left\langle-\frac{p^i}{2}+\frac{jp^i}{q-1}\right\rangle\right)\Gamma_p\left(\left\langle\frac{p^i}{4}-\frac{jp^i}{q-1}\right\rangle\right)\Gamma_p\left(\left\langle\frac{3p^i}{4}-\frac{jp^i}{q-1}\right\rangle\right)}{\Gamma_p\left(\left\langle-\frac{p^i}{2}\right\rangle\right)\Gamma_p\left(\left\langle\frac{p^i}{4}\right\rangle\right)\Gamma_p\left(\left\langle\frac{3p^i}{4}\right\rangle\right)}.
\end{align*}
Substituting the expressions for $\alpha_j$ and $M_j$ in \eqref{eq--2}, we obtain
\begin{align*}
	E_3=-\varphi(x)\sum_{j=0}^{q-2}\left((-p)^r\prod_{i=0}^{r-1}\Gamma_p\left(\left\langle\frac{p^i}{2}\right\rangle\right)^2\right)(-p)^{\beta_j}N_j\overline{\omega}^{j}\left(\frac{1}{x}\right).
\end{align*}
Using Gross-Koblitz formula and Lemma \ref{lemma2_1} for $g(\varphi)^2$, we obtain 
\begin{align}\label{new eq--1}
	q\varphi(-1)=(-p)^r\prod_{i=0}^{r-1}\Gamma_p\left(\left\langle\frac{p^i}{2}\right\rangle\right)^2.
\end{align}
Therefore, we have
\begin{align*}
	E_3=-q\varphi(-x)\sum_{j=0}^{q-2}(-p)^{\beta_j}N_j\overline{\omega}^{j}\left(\frac{1}{x}\right)=q(q-1)\varphi(-x)\cdot{_2}G_2\left[\begin{array}{cc}
		\frac{1}{4}, & \frac{3}{4}\vspace*{0.05cm} \\
		0, & \frac{1}{2}
	\end{array}|\frac{1}{x}\right]_q.
\end{align*}
\end{proof}
We now prove Theorem \ref{MT-1}.
\begin{proof}[Proof of Theorem \ref{MT-1}]
	Let $P(x,y):=ay^2+bx^2+cxy-d-ex^2y^2-fx^3y$. Then, using the identity 
	\begin{align*}
		\sum_{z\in \mathbb{F}_q} \theta(zP(x,y))= \left\{
		\begin{array}{ll}
			q, & \hbox{if $P(x,y)=0$;} \\
			0, & \hbox{if $P(x,y)\neq 0$,}
		\end{array}
		\right.
	\end{align*}
	we obtain
	\begin{align}\label{eq-13}
		q\cdot \#C_{a,b,c,d,e,f}(\mathbb{F}_q)&= \sum_{z\in\mathbb{F}_q}\sum_{x,y \in\mathbb{F}_q}\theta(zP(x,y))\nonumber\\
		&=q^2 +\sum_{z\in\mathbb{F}_q^\times}\theta(-dz)+\sum_{z,x\in\mathbb{F}_q^\times}\theta(bx^2z-dz)+\sum_{z,y\in\mathbb{F}_q^\times}\theta(ay^2z-dz)\notag\\
		& +\sum_{z,x,y \in\mathbb{F}_q^{\times}}\theta(azy^2+bzx^2+czxy-dz-ezx^2y^2-fzx^3y)\notag\\
		&=q^2+A+B+C+D,
	\end{align}
where 
\begin{align*}
	A&:=\sum_{z\in\mathbb{F}_q^\times}\theta(-dz),\\
	B&:=\sum_{z,x\in\mathbb{F}_q^\times}\theta(bx^2z-dz),\\
	C&:=\sum_{z,y\in\mathbb{F}_q^\times}\theta(ay^2z-dz),\\
	D&:=\sum_{z,x,y \in\mathbb{F}_q^{\times}}\theta(azy^2+bzx^2+czxy-dz-ezx^2y^2-fzx^3y).
\end{align*} 
Employing Lemma \ref{lemma2_02} yields
\begin{align*}
	A= \frac{1}{q-1}\sum_{\chi\in\widehat{\mathbb{F}_{q}^{\times}}}g(\overline{\chi})\chi(-d)\sum_{z\in\mathbb{F}_q^\times}\chi(z).
\end{align*}
Using \eqref{eq-1}, the sum is nonzero if and only if $\chi=\varepsilon$ and then using the fact that $g(\varepsilon)=-1$, we obtain
\begin{align}\label{eq--7}
	A=-1.
\end{align}
Using Lemma \ref{lemma2_02} for $B$, we deduce that
\begin{align*}
	B= \frac{1}{(q-1)^2}\sum_{\chi,\psi\in\widehat{\mathbb{F}_{q}^{\times}}}g(\overline{\chi})g(\overline{\psi})\chi(b)\psi(-d)\sum_{z\in\mathbb{F}_q^\times}\psi\chi(z)\sum_{x\in\mathbb{F}_q^\times}\chi^2(x).
\end{align*}
The inner sum is nonzero if and only if $\psi=\overline{\chi}$, and we obtain
\begin{align*}
	B=\frac{1}{q-1}\sum_{\chi\in\widehat{\mathbb{F}_{q}^{\times}}}g(\overline{\chi})g(\chi)\chi\left(-\frac{b}{d}\right)\sum_{x\in\mathbb{F}_q^\times}\chi^2(x).
\end{align*}
Using \eqref{eq-1}, the inner sum is nonzero if and only if $\chi$ is $\varphi$ or $\varepsilon$. Hence, we have
\begin{align*}
	B=g(\varepsilon)^2 + g(\varphi)^2\varphi(-bd).
\end{align*}
Lemma \ref{lemma2_1} yields
\begin{align}\label{eq--8}
B=1+q\varphi(bd).
\end{align}
Next, we simplify the expression for $C$. Again, using Lemma \ref{lemma2_02}, we have
\begin{align*}
	C= \frac{1}{(q-1)^2}\sum_{\chi,\psi\in\widehat{\mathbb{F}_{q}^{\times}}}g(\overline{\chi})g(\overline{\psi})\chi(a)\psi(-d)\sum_{z\in\mathbb{F}_q^\times}\psi\chi(z)\sum_{y\in\mathbb{F}_q^\times}\chi^2(y).
\end{align*}
Orthogonality relation \eqref{eq-1} yields
\begin{align*}
	C=\frac{1}{q-1}\sum_{\chi\in\widehat{\mathbb{F}_{q}^{\times}}}g(\overline{\chi})g(\chi)\chi\left(\frac{-a}{d}\right)\sum_{y\in\mathbb{F}_q^\times}\chi^2(y).
\end{align*}
Using \eqref{eq-1}, the inner sum is nonzero if and only if $\chi$ is $\varphi$ or $\varepsilon$, and hence
\begin{align*}
	C=g(\varepsilon)^2 + g(\varphi)^2\varphi(-ad).
\end{align*}
Using Lemma \ref{lemma2_1}, we deduce that
\begin{align}\label{eq--9}
C=1+q\varphi(ad).
\end{align}
We rewrite the expression for $D$ using Lemma \ref{lemma2_02} as follows  
\begin{align*}
	D=\frac{1}{(q-1)^6}\sum_{\chi,\eta, \xi, \psi,\Gamma, \gamma\in\widehat{\mathbb{F}_{q}^{\times}}}g(\overline{\chi})g(\overline{\eta})g(\overline{\xi})g(\overline{\psi})g(\overline{\Gamma})g(\overline{\gamma})\sum_{x,y,z \in\mathbb{F}_q^\times}\chi(ay^2z)\eta(bx^2z)\\
	\times\xi(cxyz)\psi(-dz)\Gamma(-ex^2y^2z)\gamma(-fx^3yz)\\
	=\frac{1}{(q-1)^6}\sum_{\chi,\eta, \xi, \psi,\Gamma, \gamma\in\widehat{\mathbb{F}_{q}^{\times}}}g(\overline{\chi})g(\overline{\eta})g(\overline{\xi})g(\overline{\psi})g(\overline{\Gamma})g(\overline{\gamma})\chi(a)\eta(b)\xi(c)\psi(-d)\Gamma(-e)\\
	\times\gamma(-f)\sum_{x\in\mathbb{F}_q^\times}\eta^2\xi\Gamma^2\gamma^3(x)\sum_{y\in\mathbb{F}_q^\times}\chi^2\xi\Gamma^2\gamma(y)\sum_{z\in\mathbb{F}_q^\times}\chi\eta\xi\psi\Gamma\gamma(z).
\end{align*}
Using the orthogonality relation \eqref{eq-1} for the summation running over $z$ we obtain that the last inner sum is nonzero if and only if $\xi=\overline{\chi\eta\psi\Gamma\gamma}$ and hence we have
\begin{align*}
	D=\frac{1}{(q-1)^5}\sum_{\chi,\eta, \psi,\Gamma, \gamma\in\widehat{\mathbb{F}_{q}^{\times}}}g(\overline{\chi})g(\overline{\eta})g(\chi\eta\psi\Gamma\gamma)g(\overline{\psi})g(\overline{\Gamma})g(\overline{\gamma})\chi\left(\frac{a}{c}\right)\eta\left(\frac{b}{c}\right)\psi\left(\frac{-d}{c}\right)\\
	\times\Gamma\left(\frac{-e}{c}\right)\gamma\left(\frac{-f}{c}\right)\sum_{x\in\mathbb{F}_q^\times}\eta\Gamma\gamma^2\overline{\chi\psi}(x)\sum_{y\in\mathbb{F}_q^\times}\chi\Gamma\overline{\eta\psi}(y).
\end{align*} 
Similarly, the inner sum running over $y$ is nonzero if and only if $\Gamma=\overline{\chi}\eta\psi$, and hence
\begin{align*}
	D=\frac{1}{(q-1)^4}\sum_{\chi,\eta, \psi, \gamma\in\widehat{\mathbb{F}_{q}^{\times}}}g(\overline{\chi})g(\overline{\eta})g(\eta^2\psi^2\gamma)g(\overline{\psi})g(\chi\overline{\eta\psi})g(\overline{\gamma})\chi\left(\frac{-a}{e}\right)\eta\left(\frac{-be}{c^2}\right)\psi\left(\frac{de}{c^2}\right)\\
	\times\gamma\left(\frac{-f}{c}\right)\sum_{x\in\mathbb{F}_q^\times}\eta^2\gamma^2\overline{\chi}^2(x).
\end{align*}
The inner sum is nonzero if and only if $\chi=\gamma\eta$ or $\chi=\gamma\eta\varphi$. Therefore, using the condition $af=ce$, we have
$$D=D_1+D_2,$$
where 
\begin{align*}
	&D_1:=\frac{1}{(q-1)^3}\sum_{\eta, \psi, \gamma\in\widehat{\mathbb{F}_{q}^{\times}}}g(\overline{\eta\gamma})g(\overline{\eta})g(\eta^2\psi^2\gamma)g(\overline{\psi})g(\gamma\overline{\psi})g(\overline{\gamma})\eta\left(\frac{ab}{c^2}\right)\psi\left(\frac{de}{c^2}\right),\\
&D_2:=\frac{\varphi(-ae)}{(q-1)^3}\sum_{\eta, \psi, \gamma\in\widehat{\mathbb{F}_{q}^{\times}}}g(\varphi\overline{\eta\gamma})g(\overline{\eta})g(\eta^2\psi^2\gamma)g(\overline{\psi})g(\varphi\gamma\overline{\psi})g(\overline{\gamma})\eta\left(\frac{ab}{c^2}\right)\psi\left(\frac{de}{c^2}\right).
\end{align*}
We first simplify the expression for $D_1$. We rewrite $D_1$ as follows
\begin{align*}
	D_1=\frac{1}{(q-1)^3} \sum_{\eta, \psi\in \widehat{\mathbb{F}_q^\times}} g(\overline{\eta})g(\overline{\psi})\eta\left(\frac{ab}{c^2}\right)\psi\left(\frac{de}{c^2}\right) \sum_{\gamma\in \widehat{\mathbb{F}_q^\times}} g(\overline{\eta\gamma})g({\eta}^2{\psi}^2\gamma)g(\gamma\overline{\psi})g(\overline{\gamma}).
\end{align*}
Using  Theorem \ref{thrm--3} with $A=\eta^2\psi^2$, $B=\overline{\eta}$, $C=\overline{\psi}$, and $D=\varepsilon$, we obtain
\begin{align}
	D_1&=\frac{1}{(q-1)^2} \sum_{\eta, \psi\in \widehat{\mathbb{F}_q^\times}} \left[\frac{g(\eta{\psi}^2)g({\eta}^2{\psi}^2)g(\overline{\eta\psi})g(\overline{\psi})}{g(\eta\psi)}+q(q-1)\psi(-1)\delta(\eta\psi)\right]\notag\\
	&\hspace*{0.5cm}\times g(\overline{\eta})g(\overline{\psi})\eta\left(\frac{ab}{c^2}\right)\psi\left(\frac{de}{c^2}\right)\label{new eq-1}\\
	&=G_1+H_1\label{new eq-4},
\end{align}
where
\begin{align*}
	G_1&:=\frac{1}{(q-1)^2} \sum_{\eta, \psi\in \widehat{\mathbb{F}_q^\times}} \frac{g(\overline{\eta})g(\overline{\psi})g(\eta{\psi}^2)g({\eta}^2{\psi}^2)g(\overline{\eta\psi})g(\overline{\psi})}{g(\eta\psi)}\eta\left(\frac{ab}{c^2}\right)\psi\left(\frac{de}{c^2}\right),\\
	H_1&:= \frac{q}{q-1} \sum_{\eta, \psi\in \widehat{\mathbb{F}_q^\times}} g(\overline{\eta})g(\overline{\psi})\eta\left(\frac{ab}{c^2}\right)\psi\left(\frac{de}{c^2}\right)\psi(-1)\delta(\eta\psi).
\end{align*}
Now, we calculate $G_1$ and $H_1$ individually. Using \eqref{delta-1}, $H_1$ is nonzero if and only if $\psi=\overline{\eta}$ and hence we deduce that
\begin{align}\label{eq--10}
	H_1&=\frac{q}{q-1}\sum_{\eta\in\widehat{\mathbb{F}_{q}^{\times}}}
	g(\overline{\eta})g(\eta)\eta\left(-\frac{ab}{de}\right)\nonumber\\	&=\frac{q}{q-1}\sum_{\eta\in\widehat{\mathbb{F}_{q}^{\times}}}
	(q\eta(-1)-(q-1)\delta(\eta))\eta\left(-\frac{ab}{de}\right)\notag\\
	&=q^2\delta\left(1-\frac{ab}{de}\right)-q,
\end{align}
where we use Lemma \ref{lemma2_1} and \eqref{eq-2} to deduce the last two equalities.  Using Davenport-Hasse relation with $m=2$ and $A=\psi\eta$, we have
\begin{align}\label{eq--11}
g(\eta^2\psi^2)=\frac{g(\psi\eta)g(\psi\eta\varphi)(\psi\eta)(4)}{g(\varphi)}.
\end{align}
Substituting the expression for $g(\eta^2\psi^2)$ from \eqref{eq--11} in $G_1$, we obtain
\begin{align*}
	G_1=\frac{1}{(q-1)^2} \sum_{\eta, \psi\in \widehat{\mathbb{F}_q^\times}} \frac{g(\overline{\eta})g(\overline{\psi})g(\eta{\psi}^2)g({\eta}{\psi})g(\varphi\eta\psi)g(\overline{\eta\psi})g(\overline{\psi})}{g(\eta\psi)g(\varphi)}\eta\left(\frac{4ab}{c^2}\right)\psi\left(\frac{4de}{c^2}\right).
\end{align*}
Multiplying the numerator and denominator by $q^2\psi(-1)g(\psi^2)$ and using the fact that $\eta(-1)^2=1$, we have
\begin{align*}
	G_1=\frac{q^2}{(q-1)^2} \sum_{\eta, \psi\in \widehat{\mathbb{F}_q^\times}} \left(\frac{\eta(-1)g(\eta{\psi}^2)g(\overline{\eta})}{qg({\psi}^2)}\right)\left(\frac{\eta\psi(-1)g(\varphi\eta\psi)g(\overline{\eta\psi})}{qg(\varphi)}\right)g({\psi}^2)g(\overline{\psi})^2\\
	\times\eta\left(\frac{4ab}{c^2}\right)\psi\left(\frac{-4de}{c^2}\right).
\end{align*}
Using \eqref{eq-0} and Lemma \ref{lemma2_2}, we obtain
\begin{align}\label{new eq-3}
	G_1&=\frac{q^2}{(q-1)^2} \sum_{\eta, \psi\in \widehat{\mathbb{F}_q^\times}} \left[{{\psi}^2\eta \choose \eta}-\frac{q-1}{q}\delta({\psi}^2)\right]{{\psi\varphi\eta}\choose {\psi\eta}}g({\psi}^2)g(\overline{\psi})^2\notag\\
	&\hspace*{3cm}\times \eta\left(\frac{4ab}{c^2}\right)\psi\left(\frac{-4de}{c^2}\right)\notag\\
	&=E_1+F_1,
\end{align}
where
\begin{align*}
	E_1&:=\frac{q^2}{(q-1)^2} \sum_{\eta, \psi\in \widehat{\mathbb{F}_q^\times}}{{\psi}^2\eta \choose \eta}{{\psi\varphi\eta}\choose {\psi\eta}} g({\psi}^2)g(\overline{\psi})^2\eta\left(\frac{4ab}{c^2}\right)\psi\left(\frac{-4de}{c^2}\right),\\
	F_1&:=\frac{-q}{q-1} \sum_{\eta, \psi\in \widehat{\mathbb{F}_q^\times}}{{\psi\varphi\eta}\choose {\psi\eta}} g({\psi}^2)g(\overline{\psi})^2\eta\left(\frac{4ab}{c^2}\right)\psi\left(\frac{-4de}{c^2}\right)\delta({\psi}^2).
\end{align*}
Firstly, we calculate $F_1$. Using the fact that $g(\varepsilon)=-1$ and $F_1$ is nonzero if and only if $\psi=\varepsilon$ or $\psi=\varphi$, we have
\begin{align*}
	F_1&=\frac{q}{q-1} \sum_{\eta\in \widehat{\mathbb{F}_q^\times}}{{\eta}\choose {\varphi\eta}} g(\varphi)^2\eta\left(\frac{4ab}{c^2}\right)\varphi\left(\frac{-4de}{c^2}\right)+\frac{q}{q-1} \sum_{\eta\in \widehat{\mathbb{F}_q^\times}} {{\varphi\eta}\choose {\eta}}\eta\left(\frac{4ab}{c^2}\right)\\
	&=\frac{q^2}{q-1} \sum_{\eta\in \widehat{\mathbb{F}_q^\times}}{{\eta}\choose {\varphi\eta}} \eta\left(\frac{4ab}{c^2}\right)\varphi\left(\frac{4de}{c^2}\right)+\frac{q}{q-1} \sum_{\eta\in \widehat{\mathbb{F}_q^\times}}{{\varphi\eta}\choose {\eta}} \eta\left(\frac{4ab}{c^2}\right),
\end{align*}
 where we use Lemma \ref{lemma2_1} to obtain the last expression. Using \eqref{eq-0} and $\varphi(\frac{4}{c^2})=1$, we have
\begin{align}\label{eq--12}
  F_1&=\frac{1}{q-1} \sum_{\eta\in \widehat{\mathbb{F}_q^\times}} \left(q\varphi(-de)J(\eta, \varphi\overline{\eta})+J(\varphi\eta,\overline{\eta})\right)\eta\left(-\frac{4ab}{c^2}\right)\nonumber\\
 &=\frac{1}{q-1} \sum_{\eta\in \widehat{\mathbb{F}_q^\times}} \left(q\varphi(-de) \sum_{z\in\mathbb{F}_q}\eta(z)\varphi\overline{\eta}(1-z)+\sum_{z\in\mathbb{F}_q}\varphi\eta(z)\overline{\eta}(1-z)\right)\eta\left(-\frac{4ab}{c^2}\right)\nonumber\\
 &=\frac{q\varphi(-de)}{q-1} \sum_{z\in\mathbb{F}_q\backslash\{1\}}\varphi(1-z)\sum_{\eta\in \widehat{\mathbb{F}_q^\times}} \eta\left(-\frac{4abz}{c^2(1-z)}\right)\nonumber\\
 &\hspace*{0.5cm}+\frac{1}{q-1}\sum_{z\in\mathbb{F}_q\backslash\{1\}}\varphi(z) \sum_{\eta\in \widehat{\mathbb{F}_q^\times}} \eta\left(-\frac{4abz}{c^2(1-z)}\right).
\end{align}
From \eqref{eq-2}, we have $F_1$ is nonzero if and only if $-\frac{4abz}{(1-z)c^2}=1$, i.e., $z=\frac{c^2}{c^2-4ab}$. Putting $z=\frac{c^2}{c^2-4ab}$ in \eqref{eq--12}, we obtain
\begin{align*}
	F_1&=q\varphi\left(-de\right)\varphi\left(-4ab(c^2-4ab)\right)+\varphi(c^2-4ab)\\
	&=q\varphi(abde(c^2-4ab))+\varphi(c^2-4ab).
\end{align*}
 In the case of $E_1$, using Theorem \ref{greene-def}, we have
\begin{align}\label{new eq-5}
	E_1&=\frac{q}{q-1} \sum_{ \psi\in \widehat{\mathbb{F}_q^\times}} g({\psi}^2)g(\overline{\psi})^2{_{2}}F_1\left(\begin{array}{cc}
		{\psi}^2 & {\varphi\psi} \\
		& \psi
	\end{array}|\frac{4ab}{c^2}
	\right) \psi\left(\frac{-4de}{c^2}\right).
\end{align}
From the transformation identity given in Theorem \ref{thrm--4}, we deduce that
\begin{align*}
	E_1 &=\frac{q}{q-1} \sum_{ \psi\in \widehat{\mathbb{F}_q^\times}} g({\psi}^2)g(\overline{\psi})^2{_{2}}F_1\left(\begin{array}{cc}
		\overline{\psi} & {\varphi} \\
		& \psi
	\end{array}|\frac{4ab}{c^2}
	\right) \psi\left(\frac{4de}{c^2}\right){\varphi{\overline{\psi}}^2}\left(1-\frac{4ab}{c^2}\right)\\
	&=\frac{q}{q-1}\varphi\left(c^2-4ab\right) \sum_{ \psi\in \widehat{\mathbb{F}_q^\times}} g({\psi}^2)g(\overline{\psi})^2{_{2}}F_1\left(\begin{array}{cc}
		\overline{\psi} & {\varphi} \\
		& \psi
	\end{array}|\frac{4ab}{c^2}
	\right)\psi\left(\frac{4dec^2}{(c^2-4ab)^2}\right).
\end{align*}
Using Definition \ref{char-def}, we first express the ${_2}F_1(\cdots)$ hypergeometric function as a character sum, and then employing Davenport-Hasse relation with $m=2$ and $A=\psi^2$, we have
\begin{align*}
	E_1&=\frac{q\varphi(c^2-4ab)}{q-1} \sum_{ \psi\in \widehat{\mathbb{F}_q^\times}} g({\psi}^2)g(\overline{\psi})^2\psi\left(\frac{4dec^2}{(c^2-4ab)^2}\right)\frac{\varphi\psi(-1)}{q}\\
	&\hspace*{0.5cm}\times\sum_{y\in\mathbb{F}_q}\varphi(y){\varphi\psi}(1-y)\psi\left(1-\frac{4aby}{c^2}\right)\\
	&=\frac{\varphi(c^2-4ab)}{q-1} \sum_{ \psi\in \widehat{\mathbb{F}_q^\times}}\frac{g(\psi)g(\varphi\psi)}{g(\varphi)}g(\overline{\psi})^2\sum_{y\in\mathbb{F}_q}\varphi(y(y-1))\psi{\left(-u_y\right)},
\end{align*}
where $u_y=\frac{16de(1-y)(c^2-4aby)}{(c^2-4ab)^2}$. Applying  Lemma \ref{lemma2_1}, we have
\begin{align*}
	E_1&=\frac{\varphi(c^2-4ab) }{q-1}\sum_{ \psi\in \widehat{\mathbb{F}_q^\times}}\left(q\psi(-1)-(q-1)\delta(\psi)\right)\frac{g(\overline{\psi})g(\varphi\psi)}{g(\varphi)}\sum_{y\in\mathbb{F}_q}\varphi(y(y-1))\psi{\left(-u_y\right)}\\
	&=M_1+N_1,
\end{align*}
where
\begin{align*}
	M_1 &:= \frac{q\varphi(c^2-4ab) }{q-1}\sum_{ \psi\in \widehat{\mathbb{F}_q^\times}}\frac{g(\overline{\psi})g(\varphi\psi)}{g(\varphi)}\sum_{y\in\mathbb{F}_q}\varphi{(y(y-1))}\psi{\left(u_y\right)},\\
	N_1&:= -\varphi(c^2-4ab) \sum_{ \psi\in \widehat{\mathbb{F}_q^\times}}\frac{g(\overline{\psi})g(\varphi\psi)}{g(\varphi)}\sum_{y\in\mathbb{F}_q}\varphi{(y(y-1))}\psi{\left(-u_y\right)}\delta(\psi).
\end{align*}
Now, we calculate $N_1$. Clearly, $N_1$ is nonzero only if $\psi=\varepsilon$. Thus, using the fact that  $g(\varepsilon)=-1$, we have
\begin{align*}
	N_1&=\varphi(c^2-4ab)  \sum_{y\in\mathbb{F}_q}\varphi{(y(y-1))}\varepsilon{\left(-u_y\right)}\\
	&=\varphi(c^2-4ab)\sum_{y\in\mathbb{F}_q\backslash\{\frac{c^2}{4ab},1\}}\varphi(y(y-1)),
\end{align*}
where we use the fact that $\varepsilon(-u_y)=0$ if $y=\frac{c^2}{4ab}$ or $y=1$, and $\varepsilon(-u_y)=1$ otherwise to obtain the last expression. Now, adding and subtracting the term under the summation for $y=\frac{c^2}{4ab}$ and $y=1$, we have
\begin{align*}
	N_1&=\varphi(c^2-4ab)\sum_{y\in\mathbb{F}_q}\varphi(y(y-1))-\varphi(c^2-4ab)\varphi\left(\frac{c^2}{4ab}\left(\frac{c^2}{4ab}-1\right)\right).
\end{align*}
We have $\sum_{y\in \mathbb{F}_q}\varphi(y(y-1))=\varphi(-1)J(\varphi, \varphi)=-1$. Hence, 
\begin{align*}
	N_1=-\varphi(c^2-4ab)-1.
\end{align*}
Using Lemma \ref{lemma2_2} in the expression of $M_1$, we obtain
\begin{align*}
	M_1 &=\frac{q\varphi(c^2-4ab)}{q-1} \sum_{ \psi\in \widehat{\mathbb{F}_q^\times}}J(\overline{\psi}, \varphi\psi)\sum_{y\in\mathbb{F}_q}\varphi{(y(y-1))}\psi(u_y)\\
	&=\frac{q\varphi(c^2-4ab)}{q-1} \sum_{ \psi\in \widehat{\mathbb{F}_q^\times}}\sum_{z\in\mathbb{F}_q}\overline{\psi}(z){\psi\varphi}(1-z)\sum_{y\in\mathbb{F}_q}\varphi{(y(y-1))}\psi(u_y)\\
	&=\frac{q\varphi(c^2-4ab)}{q-1}\sum_{z\in\mathbb{F}_q^\times}\varphi(1-z) \sum_{y\in\mathbb{F}_q}\varphi{(y(y-1))}\sum_{ \psi\in \widehat{\mathbb{F}_q^\times}}{\psi}\left(\frac{(1-z)u_y}{z}\right).
\end{align*}
	For $y\in\mathbb{F}_q$ such that $u_y=0$ or $u_y=-1$, there does not exist any $z\in\mathbb{F}_q^\times$ such that $\frac{(1-z)u_y}{z}=1$. Therefore, the inner sum is zero by orthogonality relation \eqref{eq-2}. For $y\in\mathbb{F}_q$ such that $u_y\neq0$ and $u_y\neq-1$, the inner sum is nonzero if and only if $z=\frac{u_y}{u_y+1}$. Therefore, \eqref{eq-2} yields 
\begin{align*}
	M_1 &=q\varphi(c^2-4ab)\sum_{}{'}\varphi(y(y-1)(1+u_y)),
\end{align*}
where $\sum_{}{'}$ is over $y\in\mathbb{F}_q$ such that $u_y\neq0$ and $u_y\neq-1$. Adding and subtracting the terms under the summation for $y$ for which $u_y=-1$ and using the fact that $u_y=0$ if and only if $y=1$ or $y=\frac{c^2}{4ab}$, we obtain
\begin{align*}
	M_1 &=q\varphi(c^2-4ab)\sum_{y\in\mathbb{F}_q\backslash\{1,\frac{c^2}{4ab}\}}\varphi(y(y-1)(1+u_y)).
\end{align*}
Adding and subtracting the terms under the summation for $y=1$ and $y=\frac{c^2}{4ab}$, we deduce that
\begin{align*}
	M_1 &=q\varphi(c^2-4ab)\sum_{y\in\mathbb{F}_q}\varphi(y(y-1)(1+u_y))-q.
\end{align*}
Therefore, 
\begin{align}\label{eq--15}
	D_1&=H_1+F_1+M_1+N_1\notag\\ &=q\varphi(c^2-4ab)\left(\sum_{y\in\mathbb{F}_q}\varphi(y(y-1)(1+u_y))+\varphi(abde)\right)+q^2\delta\left(1-\frac{ab}{de}\right)-2q-1.
\end{align}
Next, we evaluate $D_2$. We rewrite $D_2$ as follows
\begin{align*}
	D_2=\frac{\varphi(-ae)}{(q-1)^2}\sum_{\eta, \psi\in\widehat{\mathbb{F}_{q}^{\times}}}\left(\frac{1}{q-1}\sum_{\gamma\in\widehat{\mathbb{F}_{q}^{\times}}}g(\eta^2\psi^2\gamma)g(\varphi\overline{\eta\gamma})g(\varphi\overline{\psi}\gamma)g(\overline{\gamma})\right)g(\overline{\eta})g(\overline{\psi})\\
	\times\eta\left(\frac{ab}{c^2}\right)\psi\left(\frac{de}{c^2}\right).
\end{align*}
Using Theorem \ref{thrm--3} with $A=\eta^2\psi^2$, $B=\varphi\overline{\eta}$, $C=\varphi\overline{\psi}$, and $D=\varepsilon$, we obtain
\begin{align}
	D_2&=\frac{\varphi(-ae)}{(q-1)^2}\sum_{\eta, \psi\in\widehat{\mathbb{F}_{q}^{\times}}}\left[\frac{g(\overline{\psi\eta})g(\overline{\psi}\varphi)g(\eta\psi^2\varphi)g(\eta^2\psi^2)}{g(\eta\psi)}+q(q-1)\psi\varphi(-1)\delta(\eta\psi)\right]\notag\\
	&\hspace*{0.5cm}\times g(\overline{\eta})g(\overline{\psi})\eta\left(\frac{ab}{c^2}\right)\psi\left(\frac{de}{c^2}\right)\label{new eq-2}\\
	&=E_2+F_2\label{new eq-6},
\end{align}
where
\begin{align*}
	&E_2:=\frac{\varphi(-ae)}{(q-1)^2}\sum_{\eta, \psi\in\widehat{\mathbb{F}_{q}^{\times}}}\frac{g(\overline{\psi\eta})g(\overline{\psi}\varphi)g(\eta\psi^2\varphi)g(\eta^2\psi^2)}{g(\eta\psi)}g(\overline{\eta})g(\overline{\psi})\eta\left(\frac{ab}{c^2}\right)\psi\left(\frac{de}{c^2}\right),\\
	&F_2:=\frac{q\varphi(ae)}{q-1}\sum_{\eta, \psi\in\widehat{\mathbb{F}_{q}^{\times}}}
	g(\overline{\eta})g(\overline{\psi})\eta\left(\frac{ab}{c^2}\right)\psi\left(\frac{-de}{c^2}\right)\delta(\eta\psi).
\end{align*}
Note that $F_2=\varphi(ae)H_1$. Therefore, \eqref{eq--10} yields 
\begin{align*}
	F_2=q^2\varphi(ae)\delta\left(1-\frac{ab}{de}\right)-q\varphi(ae).
\end{align*}
We now simplify the expression for $E_2$. Substituting the value of $g(\eta^2\psi^2)$ from \eqref{eq--11} in $E_2$ gives
\begin{align*}
	E_2=\frac{\varphi(-ae)}{(q-1)^2}\sum_{ \psi\in\widehat{\mathbb{F}_{q}^{\times}}}g(\overline{\psi})g(\overline{\psi}\varphi)\psi\left(\frac{4de}{c^2}\right)\sum_{\eta\in\widehat{\mathbb{F}_{q}^{\times}}}g(\eta\psi^2\varphi)g(\overline{\eta})\frac{g(\overline{\psi\eta})g(\psi\eta\varphi)}{g(\varphi)}\eta\left(\frac{4ab}{c^2}\right).
\end{align*}
Multiplying the numerator and denominator with $q^2g(\psi^2\varphi)\psi(-1)$ and using the fact that $\eta(-1)^2=1$, we obtain
\begin{align*}
	E_2&=\frac{q^2\varphi(-ae)}{(q-1)^2}\sum_{ \psi\in\widehat{\mathbb{F}_{q}^{\times}}}g(\overline{\psi})g(\overline{\psi}\varphi)g(\psi^2\varphi)\psi\left(\frac{-4de}{c^2}\right)\sum_{\eta\in\widehat{\mathbb{F}_{q}^{\times}}}\left(\frac{g(\eta\psi^2\varphi)g(\overline{\eta})\eta(-1)}{qg(\psi^2\varphi)}\right)\\
	&\hspace*{0.5cm}\times\left(\frac{g(\overline{\psi\eta})g(\psi\eta\varphi)\eta\psi(-1)}{qg(\varphi)}\right)\eta\left(\frac{4ab}{c^2}\right)\\
	&=\frac{q^2\varphi(-ae)}{(q-1)^2}\sum_{ \psi\in\widehat{\mathbb{F}_{q}^{\times}}}g(\overline{\psi})g(\overline{\psi}\varphi)g(\psi^2\varphi)\psi\left(\frac{-4de}{c^2}\right)\sum_{\eta\in\widehat{\mathbb{F}_{q}^{\times}}}\left[{\psi^2\varphi\eta \choose \eta}-\frac{q-1}{q}\delta(\psi^2\varphi)\right]\\
	&\hspace*{0.5cm}\times{\psi\varphi\eta \choose \psi\eta}\eta\left(\frac{4ab}{c^2}\right),
\end{align*}
where we use Lemma \ref{lemma2_2} and \eqref{eq-0} to obtain the last expression. We now have 
\begin{align}\label{new eq-7}
	E_2=H_2+I_2,
\end{align}
where
\begin{align*}
	&H_2:=\frac{q^2\varphi(-ae)}{(q-1)^2}\sum_{ \psi\in\widehat{\mathbb{F}_{q}^{\times}}}g(\overline{\psi})g(\overline{\psi}\varphi)g(\psi^2\varphi)\psi\left(\frac{-4de}{c^2}\right)\sum_{\eta\in\widehat{\mathbb{F}_{q}^{\times}}}{\psi^2\varphi\eta \choose \eta}{\psi\varphi\eta \choose \psi\eta}\eta\left(\frac{4ab}{c^2}\right),\\
	&I_2:=-\frac{q\varphi(-ae)}{q-1}\sum_{ \psi\in\widehat{\mathbb{F}_{q}^{\times}}}g(\overline{\psi})g(\overline{\psi}\varphi)g(\psi^2\varphi)\psi\left(\frac{-4de}{c^2}\right)\sum_{\eta\in\widehat{\mathbb{F}_{q}^{\times}}}{\psi\varphi\eta \choose \psi\eta}\eta\left(\frac{4ab}{c^2}\right)\delta(\psi^2\varphi).
\end{align*}
We first simplify $I_2$. Using \eqref{delta-1}, $I_2$ is nonzero if and only if $\psi^2=\varphi$. Therefore, if $q\equiv3\pmod 4$, then $I_2=0$ because there does not exist any $\psi\in\widehat{\mathbb{F}_{q}^{\times}}$ such that $\psi^2=\varphi$. Next, we simplify the expression for $I_2$ when $q\equiv1\pmod4$. Clearly, when $q\equiv1\pmod4$, we have $\psi=\chi_4$ or $\psi=\overline{\chi_4}$, where $\chi_4$ is a multiplicative character of order $4$. Using the facts that $\varphi(-1)=1$ and $g(\varepsilon)=-1$, we obtain
\begin{align*}
	I_2&=\frac{q\varphi(ae)}{q-1}g(\overline{\chi_4})g(\chi_4)\chi_4\left(\frac{-4de}{c^2}\right)\sum_{\eta\in\widehat{\mathbb{F}_{q}^{\times}}}{\overline{\chi_4}\eta \choose \chi_4\eta}\eta\left(\frac{4ab}{c^2}\right)\\
	&\hspace*{0.5cm}+\frac{q\varphi(ae)}{q-1}g(\overline{\chi_4})g(\chi_4)\overline{\chi_4}\left(\frac{-4de}{c^2}\right)\sum_{\eta\in\widehat{\mathbb{F}_{q}^{\times}}}{\chi_4\eta \choose \overline{\chi_4}\eta}\eta\left(\frac{4ab}{c^2}\right)\\
	&=\frac{q^2\varphi(ae)}{q-1}\chi_4\left(\frac{4de}{c^2}\right)\sum_{\eta\in\widehat{\mathbb{F}_{q}^{\times}}}\frac{\chi_4\eta(-1)}{q}J(\overline{\chi_4}\eta,\overline{\chi_4\eta})\eta\left(\frac{4ab}{c^2}\right)\\	&\hspace*{0.5cm}+\frac{q^2\varphi(ae)}{q-1}\overline{\chi_4}\left(\frac{4de}{c^2}\right)\sum_{\eta\in\widehat{\mathbb{F}_{q}^{\times}}}\frac{\chi_4\eta(-1)}{q}J(\chi_4\eta,\chi_4\overline{\eta})\eta\left(\frac{4ab}{c^2}\right),
\end{align*}
where we use \eqref{eq-0} and Lemma \ref{lemma2_1} to obtain the last expression. Applying the definition of Jacobi sum from \eqref{eq-0}, we deduce that
\begin{align}\label{eq--13}
	I_2&=\frac{q\varphi(ae)}{q-1}\chi_4\left(\frac{-4de}{c^2}\right)\sum_{\eta\in\widehat{\mathbb{F}_{q}^{\times}}}\sum_{z\in\mathbb{F}_{q}}\overline{\chi_4}\eta(z)\overline{\chi_4\eta}(1-z)\eta\left(-\frac{4ab}{c^2}\right)\nonumber\\	&\hspace*{0.5cm}+\frac{q\varphi(ae)}{q-1}\overline{\chi_4}\left(\frac{-4de}{c^2}\right)\sum_{\eta\in\widehat{\mathbb{F}_{q}^{\times}}}\sum_{z\in\mathbb{F}_{q}}\chi_4\eta(z)\chi_4\overline{\eta}(1-z)\eta\left(-\frac{4ab}{c^2}\right)\nonumber\\
	&=\frac{q\varphi(ae)}{q-1}\chi_4\left(\frac{-4de}{c^2}\right)\sum_{z\in\mathbb{F}_{q}\backslash\{1\}}\overline{\chi_4}(z(1-z))\sum_{\eta\in\widehat{\mathbb{F}_{q}^{\times}}}\eta\left(\frac{-4abz}{c^2(1-z)}\right)\nonumber\\	&\hspace*{0.5cm}+\frac{q\varphi(ae)}{q-1}\overline{\chi_4}\left(\frac{-4de}{c^2}\right)\sum_{z\in\mathbb{F}_{q}\backslash\{1\}}\chi_4(z(1-z))\sum_{\eta\in\widehat{\mathbb{F}_{q}^{\times}}}\eta\left(\frac{-4abz}{(1-z)c^2}\right).
\end{align}
Using \eqref{eq-2}, $I_2$ is nonzero if and only if $-\frac{4abz}{(1-z)c^2}=1$, i.e., $z=\frac{c^2}{c^2-4ab}$. Putting $z=\frac{c^2}{c^2-4ab}$ in \eqref{eq--13} and then using the facts that $\chi_4(y^2)=\varphi(y)$, $\overline{\chi_4}(y^2)=\varphi(y)$, $\chi_4(y^4)=1$, and $\overline{\chi_4}(y^4)=1$ for $y\in\mathbb{F}_q^\times$, we obtain
\begin{align*}
	I_2=q\varphi(ae(c^2-4ab))\left[\chi_4\left(\frac{de}{ab}\right)+\overline{\chi_4}\left(\frac{de}{ab}\right)\right].
\end{align*}
Next, we simplify the expression for $H_2$. Theorem \ref{greene-def} yields
\begin{align}\label{new eq-8}
	H_2=\frac{q\varphi(-ae)}{q-1}\sum_{ \psi\in\widehat{\mathbb{F}_{q}^{\times}}}g(\overline{\psi})g(\overline{\psi}\varphi)g(\psi^2\varphi)\psi\left(\frac{-4de}{c^2}\right){_2}F_1\left(\begin{array}{cc}
		\psi^2\varphi & \psi\varphi \\
		& \psi
	\end{array}|\frac{4ab}{c^2}
	\right).
\end{align}
Now, using Theorem \ref{thrm--4} and the fact that $c^2-4ab\neq0$, we obtain
\begin{align*}
	H_2=\frac{q\varphi(-ae)}{q-1}\sum_{ \psi\in\widehat{\mathbb{F}_{q}^{\times}}}g(\overline{\psi})g(\overline{\psi}\varphi)g(\psi^2\varphi)\psi\left(\frac{4dec^2}{(c^2-4ab)^2}\right){_2}F_1\left(\begin{array}{cc}
		\overline{\psi}\varphi & \varphi \\
		& \psi
	\end{array}|\frac{4ab}{c^2}
	\right).
\end{align*}
Employing Definition \ref{char-def}, we deduce that
\begin{align*}
	H_2&=\frac{q\varphi(-ae)}{q-1}\sum_{ \psi\in\widehat{\mathbb{F}_{q}^{\times}}}g(\overline{\psi})g(\overline{\psi}\varphi)g(\psi^2\varphi)\psi\left(\frac{4dec^2}{(c^2-4ab)^2}\right)\\
	&\hspace*{0.5cm}\times\frac{\psi\varphi(-1)}{q}\sum_{y\in\mathbb{F}_q}\varphi(y)\varphi\psi(1-y)\psi\varphi\left(1-\frac{4aby}{c^2}\right)\\
	&=\frac{\varphi(ae)}{q-1}\sum_{y\in\mathbb{F}_q}\varphi(y(1-y)(c^2-4aby))\sum_{ \psi\in\widehat{\mathbb{F}_{q}^{\times}}}g(\overline{\psi})g(\overline{\psi}\varphi)g(\psi^2\varphi)\psi\left(\frac{-u_y}{4}\right)\\
	&=\frac{\varphi(ae)}{q-1}\sum_{y\in\mathbb{F}_q\backslash\{1,\frac{c^2}{4ab}\}}\varphi(y(1-y)(c^2-4aby))\sum_{ \psi\in\widehat{\mathbb{F}_{q}^{\times}}}g(\overline{\psi})g(\overline{\psi}\varphi)g(\psi^2\varphi)\psi\left(\frac{-u_y}{4}\right).
\end{align*}
Proposition \ref{prop-1} yields
\begin{align*}
	H_2=q\varphi(ad)\sum_{y\in\mathbb{F}_q\backslash\{1,\frac{c^2}{4ab}\}}\varphi(y)\cdot{_2}G_2\left[\begin{array}{cc}
		\frac{1}{4}, & \frac{3}{4} \vspace*{0.05cm}\\
		0, & \frac{1}{2}
	\end{array}|-\frac{1}{u_y}
	\right]_q.
\end{align*}
Substituting the final expressions for $F_2$, $I_2$, and $H_2$, we have 
\begin{align}\label{eq--16}
	D_2&=H_2+I_2+F_2\nonumber\\
	&=q\varphi(ad)\sum_{y\in\mathbb{F}_q\backslash\{1,\frac{c^2}{4ab}\}}\varphi(y)\cdot{_2}G_2\left[\begin{array}{cc}
		\frac{1}{4}, & \frac{3}{4}\vspace*{0.05cm} \\
		0, & \frac{1}{2}
	\end{array}|-\frac{1}{u_y}
	\right]_q+q^2\varphi(ae)\delta\left(1-\frac{ab}{de}\right)\nonumber\\
	&\hspace*{1cm}-q\varphi(ae)+I_2,
\end{align}
where
\begin{align*}
	I_2=\left\{
	\begin{array}{ll}
		q\varphi(ae(c^2-4ab))\left[\chi_4\left(\frac{de}{ab}\right)+\overline{\chi_4}\left(\frac{de}{ab}\right)\right] , & \hbox{if $q\equiv1\pmod4$ ;} \\
		0, & \hbox{if $q\equiv3\pmod4$.}
	\end{array}
	\right.
\end{align*}
Adding the expressions for $D_1$ and $D_2$ using \eqref{eq--15} and \eqref{eq--16}, we obtain $D$ as follows
\begin{align}\label{eq--17}
		D &=q\varphi(c^2-4ab)\left(\sum_{y\in\mathbb{F}_q}\varphi(y(y-1)(1+u_y))+\varphi(abde)\right)-1+I_2-q(2+\varphi(ae))\nonumber\\
		&\hspace*{0.5cm}+q\varphi(ad)\sum_{y\in\mathbb{F}_q\backslash\{1,\frac{c^2}{4ab}\}}\varphi(y)\cdot{_2}G_2\left[\begin{array}{cc}
			\frac{1}{4}, & \frac{3}{4} \vspace*{0.05cm}\\
			0, & \frac{1}{2}
		\end{array}|-\frac{1}{u_y}
		\right]_q+q^2\delta\left(1-\frac{ab}{de}\right)(1+\varphi(ae)).
	\end{align}
Substituting \eqref{eq--7}, \eqref{eq--8}, \eqref{eq--9}, and \eqref{eq--17} in \eqref{eq-13} completes the proof of the theorem.
\end{proof} 
\begin{proof}[Proof of Corollary \ref{special-case}]
	Since $c^2=-4ab$, so $c^2-4ab\neq 0$. Now, using $c^2=-4ab$ and $ab=de$, we have
	\begin{align}
	&\varphi(bd)=\varphi\left(\frac{-c^2d}{4a}\right)=\varphi(-ad)\label{ex-eq-1},\\
	&\varphi(ae)=\varphi\left(\frac{a^2b}{d}\right)=\varphi(bd)=\varphi\left(\frac{-c^2d}{4a}\right)=\varphi(-ad),\label{ex-eq-2}\\
	&\varphi(c^2-4ab)=\varphi(2c^2)=\varphi(2),\label{ex-eq-3}\\
	&u_y=\frac{16de(1-y)(c^2-4aby)}{(c^2-4ab)^2}=\frac{16abc^2(1-y)(1+y)}{4c^4}=y^2-1\label{ex-eq-4}.
	\end{align}
	Using \eqref{ex-eq-1}, \eqref{ex-eq-2}, \eqref{ex-eq-3}, \eqref{ex-eq-4}, and Theorem \ref{MT-1} with the condition that $ab=de$, we obtain 
	\begin{align}
	\#C_{a,b,c,d,e,f}(\mathbb{F}_q)&=2q+\varphi(ad)+\varphi(2)-2+I_2^\prime +\varphi(2)\sum_{y\in\mathbb{F}_q}\varphi(y(y-1))(1+u_y))\nonumber\\
	&\hspace*{.2cm}+q\varphi(-ad)+X\label{ex-eq-5},
	\end{align}
	where 
	\begin{align*}
	&X=\varphi(ad)\sum_{y\in\mathbb{F}_q^\times\backslash\{1,-1\} }\varphi(y) \mathbb{G}\left(-\frac{1}{u_y}
	\right),\\
	&I_2^\prime=\left\{\begin{array}{ll}
	2\varphi(2ad), & \hbox{if $q\equiv1\pmod4$;}\\
	0,& \hbox{if $q\equiv3\pmod4$.}
	\end{array}
	\right.
	\end{align*}
	Let $t:=-\frac{1}{u_y}$. It is easy to see that $t=1$ if and only if $y=0$. For $t\neq1$, we have $\frac{t-1}{t}=1+u_y=y^2$. Putting the values of $\mathbb{G}(x)$ from Theorem \ref{sv}, we obtain
	\begin{align*}
	X&=\varphi(2ad)\sum_{y\in\mathbb{F}_q^\times\backslash\{1,-1\} }\varphi(y)(\varphi(1-y)+\varphi(1+y))\\
	&=\varphi(2ad)\sum_{y\in\mathbb{F}_q }\varphi(y)(\varphi(1-y)+\varphi(1+y))-\varphi(ad)-\varphi(-ad)\\
	&=\varphi(-2ad)\sum_{y\in\mathbb{F}_q }\varphi(y(y-1))+\varphi(2ad)\sum_{y\in\mathbb{F}_q }\varphi(y(1+y))-\varphi(ad)-\varphi(-ad)\\
	&=\varphi(-2ad)\sum_{y\in\mathbb{F}_q }\varphi(y(y-1))+\varphi(2ad)\sum_{y\in\mathbb{F}_q }\varphi(-y(1-y))-\varphi(ad)-\varphi(-ad),
	\end{align*}
	where we use the change of variable $y\rightarrow-y$ in the second summation of the last expression. Now, using the fact that $\sum_{y\in\mathbb{F}_q }\varphi(y(y-1))=-1$, we obtain
	\begin{align*}
	X=-\varphi(-2ad)-\varphi(2ad)-\varphi(ad)-\varphi(-ad).
	\end{align*}
	Substituting the values of $X$ and $u_y$ in \eqref{ex-eq-5}, we have
	\begin{align*}
	\#C_{a,b,c,d,e,f}(\mathbb{F}_q)&=2q+\varphi(2)-2+I_2^\prime +\varphi(2)\sum_{y\in\mathbb{F}_q}\varphi(y(y-1))
	+q\varphi(-ad)\\
	&-\varphi(-2ad)-\varphi(2ad)-\varphi(-ad)\\
	&=2q-2+I_2^\prime+q\varphi(-ad)-\varphi(-2ad)-\varphi(2ad)-\varphi(-ad).
	\end{align*}
	Now using the fact that $\varphi(-1)=1$ if $q\equiv1\pmod4$ and $
	\varphi(-1)=-1$ if $q\equiv3\pmod4$, and substituting the value of $I_2^\prime$ we obtain the desired result.
\end{proof}
\begin{proof}[Proof of Corollary \ref{cor-1}]
	Since $\varphi(\ell)=-1$ and $\varphi((\alpha^2+4\beta)\slash \ell)=-1$, we have $\varphi(\alpha^2+4\beta)=1$.
	We take $a=1, b=-\beta, c=\alpha, d=k^2, e=k^2\ell,$ and $ f=k^2\ell\alpha $ in Theorem \ref{MT-1}.  Now, using the facts that $\varphi(\ell)=-1$ and $\varphi(\alpha^2+4\beta)=1$, we obtain the desired result.
\end{proof}
\section{Proof of Theorem \ref{MT-6}}
In Theorem \ref{MT-1}, we considered the curve $C_{a,b,c,d,e,f}$ when $c^2-4ab\neq 0$. In this section, we prove Theorem \ref{MT-6} which gives a formula for the number of $\mathbb{F}_q$-points on $C_{a,b,c,d,e,f}$ when $c^2-4ab= 0$.
\begin{proof}[Proof of Theorem \ref{MT-6}]
Following similar steps as shown in the proof of Theorem \ref{MT-1} to deduce \eqref{eq-13}, we have 
	\begin{align}\label{eq--32}
		q\cdot\#C_{a,b,c,d,e,f}=q^2+A+B+C+D_1+D_2,
	\end{align}
where $A, B, C, D_1$, and $D_2$ are as defined in \eqref{eq--7}, \eqref{eq--8}, \eqref{eq--9}, \eqref{new eq-1}, and \eqref{new eq-2}, respectively.
Using \eqref{new eq-4} and \eqref{new eq-3}, we obtain 
$$D_1=E_1+F_1+H_1,$$
where
\begin{align*}
	E_1&=\frac{q^2}{(q-1)^2} \sum_{\eta, \psi\in \widehat{\mathbb{F}_q^\times}}{{\psi}^2\eta \choose \eta}{{\psi\varphi\eta}\choose {\psi\eta}} g({\psi}^2)g(\overline{\psi})^2\eta\left(\frac{4ab}{c^2}\right)\psi\left(\frac{-4de}{c^2}\right),\\
	F_1&=\frac{-q}{q-1} \sum_{\eta, \psi\in \widehat{\mathbb{F}_q^\times}}{{\psi\varphi\eta}\choose {\psi\eta}} g({\psi}^2)g(\overline{\psi})^2\eta\left(\frac{4ab}{c^2}\right)\psi\left(\frac{-4de}{c^2}\right)\delta({\psi}^2),\\
	H_1&= \frac{q}{q-1} \sum_{\eta, \psi\in \widehat{\mathbb{F}_q^\times}} g(\overline{\eta})g(\overline{\psi})\eta\left(\frac{ab}{c^2}\right)\psi\left(\frac{de}{c^2}\right)\psi(-1)\delta(\eta\psi).
\end{align*}
We observe that $H_1=q^2\delta\left(1-\frac{ab}{de}\right)-q$, as calculated in \eqref{eq--10}. Following the similar steps as shown to deduce \eqref{eq--12} and using the condition $c^2=4ab$, we have
\begin{align*}
	F_1=\frac{q\varphi(-de)}{q-1} \sum_{z\in\mathbb{F}_q\backslash\{1\}}\varphi(1-z)\sum_{\eta\in \widehat{\mathbb{F}_q^\times}} \eta\left(\frac{z}{z-1}\right)+\frac{1}{q-1}\sum_{z\in\mathbb{F}_q\backslash\{1\}}\varphi(z) \sum_{\eta\in \widehat{\mathbb{F}_q^\times}} \eta\left(\frac{z}{z-1}\right).
\end{align*} 
Clearly, there does not exist any $z\in\mathbb{F}_q\backslash\{1\}$ such that $\frac{z}{z-1}=1$. Therefore, using \eqref{eq-2} we obtain the inner sums for both the terms are zero and hence $F_1=0$.
Using \eqref{new eq-5} the simplified expression for $E_1$ can be written as follows 
\begin{align*}
		E_1=\frac{q}{q-1} \sum_{ \psi\in \widehat{\mathbb{F}_q^\times}} g({\psi}^2)g(\overline{\psi})^2{_{2}}F_1\left(\begin{array}{cc}
		{\psi}^2 & {\varphi\psi} \\
		& \psi
	\end{array}|\frac{4ab}{c^2}
	\right) \psi\left(\frac{-4de}{c^2}\right).
\end{align*}
 Using the transformation identity given in Theorem \ref{thrm--4}, we obtain
\begin{align*}
	E_1&=\frac{q}{q-1}\sum_{\psi\in\widehat{\mathbb{F}_q^\times}}g(\psi^2)g(\overline{\psi})^2{\varphi\psi\choose\overline{\psi}}\psi\left(\frac{-4de}{c^2}\right)\\
	&=\frac{1}{q-1}\sum_{\psi\in\widehat{\mathbb{F}_q^\times}}g(\psi^2)g(\overline{\psi})^2J(\varphi\psi,\psi)\psi\left(\frac{4de}{c^2}\right)\\
	&=\frac{\varphi(-1)}{q-1}\sum_{\psi\in\widehat{\mathbb{F}_q^\times}}g(\psi^2)g(\overline{\psi})^2J(\varphi\psi,\varphi\overline{\psi}^2)\psi\left(\frac{-4de}{c^2}\right)\\
	&=\frac{\varphi(-1)}{q-1}\sum_{\psi\in\widehat{\mathbb{F}_q^\times},\psi\neq\varepsilon} g(\psi^2)g(\overline{\psi})^2J(\varphi\psi,\varphi\overline{\psi}^2)\psi\left(\frac{-4de}{c^2}\right) - \frac{\varphi(-1)}{q-1}J(\varphi,\varphi),
	\end{align*}
where we use the fact that $J(A,B)=A(-1)J(A,\overline{AB})$ for any multiplicative characters $A$ and $B$ to obtain the second last expression. Employing Lemma \ref{lemma2_2} and the fact that $J(\varphi,\varphi)=-\varphi(-1)$, we deduce that
	\begin{align*}
	E_1=\frac{\varphi(-1)}{q-1}\sum_{\psi\in\widehat{\mathbb{F}_q^\times},\psi\neq\varepsilon}g(\psi^2)g(\overline{\psi})^2\frac{g(\varphi\psi)g(\varphi\overline{\psi}^2)}{g(\overline{\psi})}\psi\left(\frac{-4de}{c^2}\right)+ \frac{1}{q-1}.
\end{align*}
Using Davenport-Hasse relation with $m=2$, $A=\psi$ and $m=2$, $A=\overline{\psi}^2$ yield
\begin{align*}
	E_1&=\frac{\varphi(-1)}{q-1}\sum_{\psi\in\widehat{\mathbb{F}_q^\times},\psi\neq\varepsilon}\frac{g(\psi)g(\varphi\psi)^2g(\overline{\psi})g(\overline{\psi}^4)}{g(\overline{\psi}^2)}\psi\left(\frac{-16^2de}{c^2}\right)+ \frac{1}{q-1}\\
	&=\frac{q\varphi(-1)}{q-1}\sum_{\psi\in\widehat{\mathbb{F}_q^\times},\psi\neq\varepsilon}\frac{g(\varphi\psi)^2g(\overline{\psi}^{4})}{g(\overline{\psi}^2)}\psi\left(\frac{16^2de}{c^2}\right)+ \frac{1}{q-1},
\end{align*}
where we use Lemma \ref{lemma2_1} to obtain the last equality. Replacing $\psi$ with $\omega^j$ , we have
\begin{align*}
	E_1= \frac{q\varphi(-1)}{q-1}\sum_{j=1}^{q-2}\frac{g(\overline{\omega}^{-j+\frac{q-1}{2}})^2g(\overline{\omega}^{4j})}{g(\overline{\omega}^{2j})}\overline\omega^j\left(\frac{c^2}{16^2de}\right)+\frac{1}{q-1}.
\end{align*}
To calculate $E_1$, we first simplify a character sum $Z(\lambda)$ for $\lambda\in\mathbb{F}_q^\times$ defined as follows
\begin{align}\label{eq--33}
	Z(\lambda):=\frac{q\varphi(-1)}{q-1}\sum_{j=1}^{q-2}\frac{g(\overline{\omega}^{-j+\frac{q-1}{2}})^2g(\overline{\omega}^{4j})}{g(\overline{\omega}^{2j})}\overline\omega^j\left(\lambda\right)+\frac{1}{q-1}.
\end{align}
Using Gross-Kobilitz formula, we obtain
\begin{align*}
Z(\lambda)&=\frac{q\varphi(-1)}{q-1}\sum_{j=1}^{q-2}(-p)^{\zeta_j}L_j\overline\omega^j\left(\lambda\right)+\frac{1}{q-1},
\end{align*}
where 
 \begin{align*}
	\zeta_j&:=\sum_{i=0}^{r-1}\left(2\left\langle\frac{p^i}{2}-\frac{jp^i}{q-1}\right\rangle+  \left\langle\frac{4jp^i}{q-1}\right\rangle-\left\langle\frac{2jp^i}{q-1}\right\rangle\right),\\
	L_j&:=\prod_{i=0}^{r-1}\frac{{\Gamma_p\left(\left\langle\frac{p^i}{2}-\frac{jp^i}{q-1}\right\rangle\right)}^2\Gamma_p\left(\left\langle\frac{4jp^i}{q-1}\right\rangle\right)}{\Gamma_p\left(\left\langle\frac{2jp^i}{q-1}\right\rangle\right)}.
\end{align*}
Using the fact that  $x=\left\lfloor x \right\rfloor+\left\langle x \right\rangle$ for all $x\in \mathbb{R}$, we obtain
\begin{align*}
	\zeta_j&=\sum_{i=0}^{r-1}\left(1-2\left\lfloor\frac{1}{2}-\frac{jp^i}{q-1}\right\rfloor-  \left\lfloor\frac{4jp^i}{q-1}\right\rfloor+\left\lfloor\frac{2jp^i}{q-1}\right\rfloor\right).
\end{align*}
For $1\leq j\leq q-2$, using Lemma \ref{lemma-3_4} with $l=4$ and $l=2$ yield 
\begin{align}\label{new eq--9}
	\zeta_j=r+\zeta_j^\prime,
\end{align}
where $\zeta_j^\prime:=-\sum_{i=0}^{r-1}\left(2\left\lfloor\left\langle\frac{p^i}{2}\right\rangle-\frac{jp^i}{q-1}\right\rfloor+ \left\lfloor\left\langle\frac{-p^i}{4}\right\rangle+\frac{jp^i}{q-1}\right\rfloor+\left\lfloor\left\langle\frac{-3p^i}{4}\right\rangle+\frac{2jp^i}{q-1}\right\rfloor\right)$. Using Lemma \ref{lemma3_2} with $t=4$ and $t=2$, and multiplying numerator and denominator by $\Gamma_p\left(\left\langle\frac{p^i}{2}\right\rangle\right)^2$, we deduce 
\begin{align}\label{new eq--10}
	L_j&=\overline{\omega}^j(4^3){\prod_{i=0}^{r-1}\frac{{\Gamma_p\left(\left\langle\frac{p^i}{2}\right\rangle\right)}^2{\Gamma_p\left(\left\langle\frac{p^i}{2}-\frac{jp^i}{q-1}\right\rangle\right)}^2\Gamma_p\left(\left\langle\frac{-p^i}{4}+\frac{jp^i}{q-1}\right\rangle\right)\Gamma_p\left(\left\langle\frac{-3p^i}{4}+\frac{jp^i}{q-1}\right\rangle\right)}{{\Gamma_p\left(\left\langle\frac{p^i}{2}\right\rangle\right)}^2\Gamma_p\left(\left\langle\frac{-p^i}{4}\right\rangle\right)\Gamma_p\left(\left\langle\frac{-3p^i}{4}\right\rangle\right)}}\notag\\
	&=\overline{\omega}^j(4^3)L_j^\prime{\prod_{i=0}^{r-1}\Gamma_p\left(\left\langle\frac{p^i}{2}\right\rangle\right)}^2,
\end{align}
where $L_j^\prime:={\prod_{i=0}^{r-1}\frac{{\Gamma_p\left(\left\langle\frac{p^i}{2}-\frac{jp^i}{q-1}\right\rangle\right)}^2\Gamma_p\left(\left\langle\frac{-p^i}{4}+\frac{jp^i}{q-1}\right\rangle\right)\Gamma_p\left(\left\langle\frac{-3p^i}{4}+\frac{jp^i}{q-1}\right\rangle\right)}{{\Gamma_p\left(\left\langle\frac{p^i}{2}\right\rangle\right)}^2\Gamma_p\left(\left\langle\frac{-p^i}{4}\right\rangle\right)\Gamma_p\left(\left\langle\frac{-3p^i}{4}\right\rangle\right)}}$. Substituting the expressions for $\zeta_j$ and $L_j$ from \eqref{new eq--9} and \eqref{new eq--10}, respectively, in \eqref{eq--33} and then  using \eqref{new eq--1}, we obtain 
\begin{align}\label{new eq--8}
	Z(\lambda)&=\frac{q^2}{q-1}\sum_{j=1}^{q-2}(-p)^{\zeta_j^\prime}L_j^\prime\overline{\omega}^j{\left(4^3\lambda\right)}+\frac{1}{q-1}\notag\\
	&=-q^2\cdot {{_2}G_2}\left[\begin{array}{cc}
		\frac{1}{2}, & \frac{1}{2}\vspace*{0.05cm} \\
		\frac{1}{4}, & \frac{3}{4}
	\end{array}|4^3\lambda
	\right]_q   -q-1.
\end{align}
Clearly, $E_1=Z\left({\frac{c^2}{16^2de}}\right)$. Therefore, the simplified expression for $D_1$ is as follows
\begin{align}\label{eq--30}
	D_1=-q^2\cdot {{_2}G_2}\left[\begin{array}{cc}
		\frac{1}{2}, & \frac{1}{2} \vspace*{0.05cm}\\
		\frac{1}{4}, & \frac{3}{4}
	\end{array}|\frac{c^2}{4de}
	\right]_q-q-1+q^2\delta\left(1-\frac{ab}{de}\right)-q.
\end{align}
Next, we simplify the expression for $D_2$. Using \eqref{new eq-6} and \eqref{new eq-7}, we have
\begin{align*}
	D_2=H_2+I_2+F_2,
\end{align*}
where 
\begin{align*}
	&F_2=\frac{q\varphi(ae)}{q-1}\sum_{\eta, \psi\in\widehat{\mathbb{F}_{q}^{\times}}}
	g(\overline{\eta})g(\overline{\psi})\eta\left(\frac{ab}{c^2}\right)\psi\left(\frac{-de}{c^2}\right)\delta(\eta\psi),\\
	&H_2=\frac{q^2\varphi(-ae)}{(q-1)^2}\sum_{ \psi\in\widehat{\mathbb{F}_{q}^{\times}}}g(\overline{\psi})g(\overline{\psi}\varphi)g(\psi^2\varphi)\psi\left(\frac{-4de}{c^2}\right)\sum_{\eta\in\widehat{\mathbb{F}_{q}^{\times}}}{\psi^2\varphi\eta \choose \eta}{\psi\varphi\eta \choose \psi\eta}\eta\left(\frac{4ab}{c^2}\right),\\
	&I_2=-\frac{q\varphi(-ae)}{q-1}\sum_{ \psi\in\widehat{\mathbb{F}_{q}^{\times}}}g(\overline{\psi})g(\overline{\psi}\varphi)g(\psi^2\varphi)\psi\left(\frac{-4de}{c^2}\right)\sum_{\eta\in\widehat{\mathbb{F}_{q}^{\times}}}{\psi\varphi\eta \choose \psi\eta}\eta\left(\frac{4ab}{c^2}\right)\delta(\psi^2\varphi).
\end{align*}
Clearly, $F_2=\varphi(ae)H_1$. Hence, $F_2=\varphi(ae)\left(q^2\delta\left(1-\frac{ab}{de}\right)-q\right)$. Following similar steps as shown to deduce \eqref{eq--13} and using the condition $c^2=4ab$, we have $I_2=0$ for $q\equiv3\pmod4$ and
\begin{align*}
	I_2=\frac{q\varphi(ae)}{q-1}\chi_4\left(\frac{-4de}{c^2}\right)\sum_{z\in\mathbb{F}_{q}\backslash\{1\}}\overline{\chi_4}(z(1-z))\sum_{\eta\in\widehat{\mathbb{F}_{q}^{\times}}}\eta\left(\frac{z}{z-1}\right)\\
	+\frac{q\varphi(ae)}{q-1}\overline{\chi_4}\left(\frac{-4de}{c^2}\right)\sum_{z\in\mathbb{F}_{q}\backslash\{1\}}\chi_4(z(1-z))\sum_{\eta\in\widehat{\mathbb{F}_{q}^{\times}}}\eta\left(\frac{z}{z-1}\right)
\end{align*}
for $q\equiv1\pmod4$. Clearly, there does not exist any $z\in\mathbb{F}_q\backslash\{1\}$ such that $\frac{z}{z-1}=1$. Therefore, using \eqref{eq-2} we obtain the inner sums for both the terms are zero and hence $I_2=0$.
Using \eqref{new eq-8}, the simplified expression for $H_2$ can be written as follows
\begin{align*}
		H_2=\frac{q\varphi(-ae)}{q-1}\sum_{ \psi\in\widehat{\mathbb{F}_{q}^{\times}}}g(\overline{\psi})g(\overline{\psi}\varphi)g(\psi^2\varphi)\psi\left(\frac{-4de}{c^2}\right){_2}F_1\left(\begin{array}{cc}
		\psi^2\varphi & \psi\varphi \\
		& \psi
	\end{array}|\frac{4ab}{c^2}
	\right).
\end{align*}
Using Theorem \ref{thrm--4} and the condition that $c^2=4ab$, we have
\begin{align}\label{new eq--5}
	H_2&=\frac{q\varphi(-ae)}{q-1}\sum_{\psi\in\widehat{\mathbb{F}_q^\times}}g(\overline{\psi})g(\overline{\psi}\varphi)g({\psi}^2\varphi)\psi^2\varphi(-1)\psi\left(\frac{-4de}{c^2}\right){\psi\varphi\choose\overline{\psi}\varphi}\notag \\
	&=\frac{\varphi(-ae)}{q-1}\sum_{\psi\in\widehat{\mathbb{F}_q^\times}}g(\overline{\psi})g(\overline{\psi}\varphi)g({\psi}^2\varphi)\psi\left(\frac{4de}{c^2}\right)J({\psi\varphi,{\psi}\varphi})\notag\\
	&=\frac{\varphi(-ae)}{q-1}\sum_{\psi\in\widehat{\mathbb{F}_q^\times}, \psi\neq\varphi,\varepsilon}g(\overline{\psi})g(\overline{\psi}\varphi)g({\psi}^2\varphi)\psi\left(\frac{4de}{c^2}\right)J({\psi\varphi,{\psi}\varphi})\notag\\&\hspace{4cm}-\frac{q\varphi(ae)}{q-1}J(\varphi,\varphi)
	-\frac{q\varphi(ad)}{q-1}J(\varepsilon,\varepsilon)\notag\\
	&=\frac{\varphi(-ae)}{q-1}\sum_{\psi\in\widehat{\mathbb{F}_q^\times}, \psi\neq\varphi,\varepsilon}g(\overline{\psi})g(\overline{\psi}\varphi)g({\psi}^2\varphi)\psi\left(\frac{4de}{c^2}\right)\frac{g(\varphi\psi)g(\varphi\psi)}{g(\psi^2)}\notag\\&\hspace{4cm}-\frac{q\varphi(ae)}{q-1}J(\varphi,\varphi)
	-\frac{q\varphi(ad)}{q-1}J(\varepsilon,\varepsilon).
\end{align}
Using Davenport-Hasse relation with $m=2$, $A=\psi^2$ and $m=2$, $A=\psi$, we have 
\begin{align}\label{new eq--6}
	g(\psi^2)=\frac{g(\psi)g(\varphi\psi)\psi(4)}{g(\varphi)},\ \ \ \ \
	g(\varphi\psi^2)=\frac{g(\psi^4)g(\varphi)\psi^2(4^{-1})}{g(\psi^2)}.
\end{align}
Substituting the values of $g(\psi^2)$ and $g(\varphi\psi^2)$ from \eqref{new eq--6} in \eqref{new eq--5} and multiplying the numerator and denominator by $g(\overline{\psi}),$ we obtain
\begin{align*}
	H_2&=\frac{\varphi(-ae)}{q-1}\sum_{\psi\in\widehat{\mathbb{F}_q^\times}, \psi\neq\varphi,\varepsilon}\frac{g(\overline{\psi})^2g(\overline{\psi}\varphi)g({\psi}\varphi)g(\psi^4)g(\varphi)^2}{g(\psi)g(\overline{\psi})g(\psi^2)}\psi\left(\frac{de}{16c^2}\right)\notag\\
	&\hspace{3cm}-\frac{q\varphi(ae)}{q-1}J(\varphi,\varphi)
	-\frac{q\varphi(ad)}{q-1}J(\varepsilon,\varepsilon).
\end{align*}
Using Lemma \ref{lemma2_1} and the facts that $J(\varphi,\varphi)=-\varphi(-1)$ and $J(\varepsilon,\varepsilon)=q-2$, we have
\begin{align*}
	H_2&=\frac{q\varphi(-ae)}{q-1}\sum_{\psi\in\widehat{\mathbb{F}_q^\times}, \psi\neq\varphi,\varepsilon}\frac{g(\overline{\psi})^2g(\psi^4)}{g(\psi^2)}\psi\left(\frac{de}{16c^2}\right)+\frac{q\varphi(-ae)}{q-1}
	-\frac{q(q-2)\varphi(ad)}{q-1}\\
	&=\frac{q\varphi(-ae)}{q-1}\sum_{\psi\in\widehat{\mathbb{F}_q^\times}}\frac{g(\overline{\psi})^2g(\psi^4)}{g(\psi^2)}\psi\left(\frac{de}{16c^2}\right)-2q\varphi(ad),
\end{align*}
 where we add and subtract the term under the summation for $\psi=\varepsilon$ and $\psi=\varphi$ to obtain the last equality. The change of variable $\psi\rightarrow\psi\varphi$ yields 
\begin{align*}
	H_2=\frac{q\varphi(-ad)}{q-1}\sum_{\psi\in\widehat{\mathbb{F}_q^\times}}\frac{g(\overline{\psi}\varphi)^2g(\psi^4)}{g(\psi^2)}\psi\left(\frac{de}{16c^2}\right)-2q\varphi(ad).
\end{align*}
Replacing $\psi$ with $\overline{\omega}^j$ and taking out the term for $j=0$, we have
\begin{align}\label{new-eqn-101}
	H_2&=\frac{q\varphi(-ad)}{q-1}\sum_{j=1}^{q-2}\frac{g(\overline{\omega}^{-j+\frac{q-1}{2}})^2g(\overline{\omega}^{4j})}{g(\overline{\omega}^{2j})}\overline{\omega}^j{\left(\frac{de}{16c^2}\right)}-2q\varphi(ad)+\frac{q^2\varphi(ad)}{q-1}. 
\end{align}
From \eqref{eq--33} and \eqref{new-eqn-101}, we have 
\begin{align*}
	H_2&=\varphi(ad)\left(Z\left(\frac{de}{16c^2}\right)-\frac{1}{q-1}\right)-2q\varphi(ad)+\frac{q^2\varphi(ad)}{q-1}\\
	&=\varphi(ad)\left(Z\left(\frac{de}{16c^2}\right)-\frac{1}{q-1}-2q+\frac{q^2}{q-1}\right)\\
	&=\varphi(ad)\left(Z\left(\frac{de}{16c^2}\right)-q+1\right).
\end{align*} 
From \eqref{new eq--8}, we obtain 
\begin{align*}
	H_2=-q^2\varphi(ad)\cdot{{_2}G_2}\left[\begin{array}{cc}
		\frac{1}{2}, & \frac{1}{2} \vspace*{0.05cm}\\
		\frac{1}{4}, & \frac{3}{4}
	\end{array}|\frac{4de}{c^2}
	\right]_q-2q\varphi(ad).
\end{align*} 
Hence,
\begin{align}\label{eq--31}
	D_2=-q^2\varphi(ad)\cdot{{_2}G_2}\left[\begin{array}{cc}
		\frac{1}{2}, & \frac{1}{2} \vspace*{0.05cm}\\
		\frac{1}{4}, & \frac{3}{4}
	\end{array}|\frac{4de}{c^2}
	\right]_q+\varphi(ae)\left(q^2\delta\left(1-\frac{ab}{de}\right)-q\right)-2q\varphi(ad).
\end{align}
Substituting the expressions from  \eqref{eq--7}, \eqref{eq--8}, \eqref{eq--9}, \eqref{eq--30}, \eqref{eq--31} in \eqref{eq--32} and using the fact that $\varphi(bd)=\varphi(ad)$, we complete the proof of the theorem.
\end{proof}
\begin{proof}[Proof of Corollary \ref{new_cor-1}]
Using conditions $ab=de$ and $c^2=4ab$ in Theorem \ref{MT-6}, we obtain
\begin{align*}
	\#C_{a,b,c,d,e,f}(\mathbb{F}_q)&=q-1-q\varphi(ad)\cdot {_2}G_2\left[\begin{array}{cc}
		\frac{1}{2}, & \frac{1}{2}\vspace*{0.05cm}\\
		\frac{1}{4}, & \frac{3}{4}
	\end{array}|1
	\right]_q-q\cdot {_2}G_2\left[\begin{array}{cc}
		\frac{1}{2}, & \frac{1}{2} \vspace*{0.05cm}\\
		\frac{1}{4}, & \frac{3}{4}
	\end{array}|1
	\right]_q\\
	&\hspace{1cm} + \left(q-1\right)(1+\varphi(ae)).
\end{align*} 
Using \cite[Lemma 5.2 (1)]{SB2} and \cite[Lemma 5.2 (2)]{SB2} we have
\begin{align*}
	\#C_{a,b,c,d,e,f}(\mathbb{F}_q)&=2q-2-\varphi(-ad)\cdot {_2}G_2\left[\begin{array}{cc}
		\frac{1}{4}, & \frac{3}{4}\vspace*{0.05cm}\\
		0, & 0
	\end{array}|1
	\right]_q-\varphi(-1)\cdot {_2}G_2\left[\begin{array}{cc}
		\frac{1}{4}, & \frac{3}{4} \vspace*{0.05cm}\\
		0, & 0
	\end{array}|1
	\right]_q\\
	&\hspace{1cm} + \varphi(ae)\left(q-1\right).
\end{align*} 
Employing \cite[Proposition 3.5 (3-7)]{BS3} and the fact that $\varphi(
-1)=1$ for $q\equiv1\pmod4$, we obtain 
\begin{align*}
	\#C_{a,b,c,d,e,f}(\mathbb{F}_q)&=2q-2+q\varphi(ad)\cdot {_{2}}F_1\left(\begin{array}{cc}
		\chi_4 & \chi_4^3 \\
		& \varepsilon
	\end{array}|1
	\right)+q\cdot{_{2}}F_1\left(\begin{array}{cc}
		\chi_4 & \chi_4^3 \\
		& \varepsilon
	\end{array}|1
	\right)\\
	&\hspace{1cm} + \varphi(ae)\left(q-1\right),
\end{align*}
where $\chi_4$ is a multiplicative character of order $4$ on $\mathbb{F}_q$. Using \cite[Theorem 4.9]{greene}, we have
\begin{align*}
	\#C_{a,b,c,d,e,f}(\mathbb{F}_q)&=2q-2+q\varphi(ad)\chi_4(-1){\overline{\chi_4}\choose\overline{\chi_4}}+q\chi_4(-1){\overline{\chi_4}\choose\overline{\chi_4}}+ \varphi(ae)\left(q-1\right).
\end{align*}
Clearly, using $ab=de$ and $c^2=4ab$, we have $\varphi(d)=\varphi(e)$, that is, $\varphi(ad)=\varphi(ae)$. Also, \eqref{eq-4} yields ${\overline{\chi_4}\choose\overline{\chi_4}}=-\frac{1}{q}$. Hence, we have
\begin{align*}
	\#C_{a,b,c,d,e,f}(\mathbb{F}_q)&=2q-2-\varphi(ad)\chi_4(-1)-\chi_4(-1)+(q-1)\varphi(ad).
\end{align*}
Using the facts that $\chi_4(-1)=1$ if $q\equiv1\pmod8$ and $\chi_4(-1)=-1$ if $q\equiv5\pmod8$, we complete the proof.
\end{proof}
\begin{proof}[Proof of Corollary \ref{cor-2}]
From Theorem \ref{thrm--5} and using the condition that $c^2g=h^2de$, we obtain 
\begin{align}
	a_q(E_{h,g})=q\cdot\varphi(-hg)\cdot{_2}G_{2}\left[\begin{array}{cc}
		\frac{1}{2},  & \frac{1}{2}\vspace*{0.07cm} \\
		\frac{1}{4}, & \frac{3}{4}
	\end{array}|\frac{4de}{c^2}
	\right]_q,\label{eq--27}\\
		a_q(E_{\frac{4}{h},\frac{1}{g}})=q\cdot\varphi(-hg)\cdot{_2}G_{2}\left[\begin{array}{cc}
		\frac{1}{2},  & \frac{1}{2}\vspace*{0.07cm} \\
		\frac{1}{4}, & \frac{3}{4}
	\end{array}|\frac{c^2}{4de}
	\right]_q\label{eq--26}.
\end{align}
Also, using the fact that $g=\frac{deh^2}{c^2}$, we have 
\begin{align}\label{neq eq--28}
	\varphi(gd)=\varphi(e).
\end{align}
Combining \eqref{eq--27}, \eqref{eq--26}, \eqref{neq eq--28}, and Theorem \ref{MT-6}, we obtain the desired result.
\end{proof}

\end{document}